\documentclass{amsart}

\usepackage{todonotes}
\usepackage{enumerate}
\usepackage[english]{babel}

\usepackage[letterpaper,top=2cm,bottom=2cm,left=3cm,right=3cm,marginparwidth=1.75cm]{geometry}

\usepackage{amssymb}
\usepackage{latexsym}
\usepackage{amsfonts}
\usepackage{mathrsfs}
\usepackage{amsmath}
\usepackage{amsthm}
\usepackage{multicol}
\usepackage{constants}
\usepackage{color}
\usepackage{bm}
\usepackage{ulem}
\usepackage[colorlinks=true, allcolors=blue]{hyperref}
\theoremstyle{definition}
\newtheorem{definition}{Definition}[section]
\newtheorem{proposition}[definition]{Proposition}
\newtheorem{theorem}[definition]{Theorem}
\newtheorem{corollary}[definition]{Corollary}
\newtheorem{lemma}[definition]{Lemma}
\newtheorem{remark}[definition]{Remark}

\def\R{{\mathbb{R}}}
\newcommand{\divergence}{\mathrm{div}\,}
\newcommand{\spt}{\mathrm{spt}\,}
\newcommand{\dist}{\mathrm{dist}\,}

\title{Generic mean curvature flow with obstacles}
\author{Tim Laux}
\address{(Tim Laux) 
Institute for Mathematics and Interdisciplinary Center for Scientific Computing, Heidelberg University\\
         Im Neuenheimer Feld 205\\
         D-69120 Heidelberg\\
         Germany}
\email{tim.laux@math.uni-heidelberg.de}

\author{Keisuke Takasao}

\address{(Keisuke Takasao) Department  of Mathematics\\
Graduate School of Science\\
        Kyoto University\\
        Kitashirakawa-Oiwakecho Sakyo\\
        Kyoto 606-8502\\
        Japan}
\email{k.takasao@math.kyoto-u.ac.jp}
\begin{document}
\maketitle

\begin{abstract}
    We study the obstacle problem associated to mean curvature flow. 
    We add to the geometric vanishing-viscosity approximation of Evans and Spruck a singular perturbation that penalizes the violation of the constraint, and pass to the limit. 
    The resulting level set formulation has unique solutions---up to fattening. 
    Extending the work of Evans and Spruck and a work by Ullrich and one of the authors, we show that generic level sets of this flow are distributional solutions of the obstacle problem.
    
    \medskip
    

     \medskip
  
    \noindent \textbf{Mathematical Subject Classification (MSC2020)}: 
    53E10 (Primary); 
    35K65; 
    35A15; 
    35D30; 
    35D40; 
    
\end{abstract}

\section{Introduction and main results}

Mean curvature flow is one of the most fundamental geometric evolution equations behind many phenomena in science and engineering. 
In this paper, we are concerned with its associated obstacle problem, in which the flow is not allowed to penetrate certain regions.
In this paper, we construct solutions to the obstacle problem via Evans--Spruck's geometric vanishing viscosity approach. 
While this process naturally gives rise to the unique viscosity solution of the associated level set formulation, the main contribution of the present paper is a finer analysis revealing the evolution law for generic level sets of this solution:
Our main result shows that almost every level set is a BV solution to the obstacle problem for mean curvature flow.

There is a rich, but still growing stock of literature on the obstacle problem for mean curvature flow.
Almeida, Chambolle, and Novaga~\cite{MR2971026} extend Chambolle's version~\cite{MR2079603} of the implicit time discretization of Almgren-Taylor-Wang~\cite{MR1205983} and Luckhaus-Sturzenhecker~\cite{MR1386964} to the case of the obstacle problem and study its convergence in the planar case $d=2$. 
Mercier and Novaga~\cite{MR3421913} studied the graphical case, proved short-time existence of classical solutions, long-time existence of viscosity solutions of the graphical mean curvature flow, and studied (subsequential) asymptotics for large times towards stationary solutions.
We also mention the PhD thesis of Logaritsch~\cite{Logaritsch} in which the implicit time discretization in case of graphs with obstacles is studied in depth. Remarkably, in this case, one can prove the convergence of area and prove a variational inequality for the limit of this discretization.
Mercier~\cite{Mercier} and Ishii et al.~\cite{MR3746189} constructed viscosity solutions for the obstacle problem without the assumption of being a graph.
In the long-term limit, the mean curvature flow with obstacle is expected to converge to the mean convex hull of the obstacle.
Spadaro \cite{MR4069612} used this idea to show an optimal regularity result for the mean convex hull via the long-time limit of the implicit time discretization of the obstacle mean curvature flow.
Based on a forced Allen--Cahn approximation, in dimensions $d=2,3$ in~~\cite{MR4335627}, one of the authors constructed global weak solutions that satisfy a version of Brakke's formulation away from the obstacles. 
Together with Nik~\cite{MR4848696}, this dimensional restriction could be removed.
While these Brakke-type solutions describe the evolution away from the obstacles, they do not give any insight on the behavior at the contact with the obstacles.
Beside the viscosity solution, which relies on the comparison principle of two-phase mean curvature flow, we develop here the first weak solution concept that fully describes the solution to the obstacle problem.

Let $\Omega\subset \mathbb{R}^d$ be a bounded domain and let $\phi , \psi \in C^1 (\Omega)$ be functions 
such that $\phi < \psi $ in $\Omega$.
We want to construct and study level set solutions to mean curvature flow with obstacles $\phi$ and $\psi$.
That means that the level set function should satisfy the constraint 
\begin{equation}\label{eq:levelset_constraint}
    \phi \leq u \leq \psi.
\end{equation}
Away from the contact sets $\{u=\phi\}$ and $\{u=\psi\}$, this formulation should encode the classical level set mean curvature flow
\begin{equation}\label{eq:levelset_interior}
\partial_t u 
    =\Big(\nabla \cdot \frac{\nabla u}{|\nabla u|}  \Big)|\nabla u|,
\end{equation}
while on the contact sets, we need to encode the inequalities
\begin{equation}\label{eq:levelset_contact}
   \partial_t u 
    \geq\Big(\nabla \cdot \frac{\nabla u}{|\nabla u|}  \Big)|\nabla u| \quad \text{on }\{u=\phi\}
\quad \text{and} \quad 
   \partial_t u 
    \leq \Big(\nabla \cdot \frac{\nabla u}{|\nabla u|}  \Big)|\nabla u| \quad \text{on } \{u=\psi\}.
\end{equation}

We construct our solutions via a modification of the geometric vanishing-viscosity approximation of Evans--Spruck~\cite{MR1100206} in a potential
\[
    \partial_t u_\varepsilon 
    =\Big(\nabla \cdot \frac{\nabla u_\varepsilon}{\sqrt{\varepsilon^2+|\nabla u_\varepsilon|^2}} - V'_\varepsilon (u_\varepsilon)\Big)
    \sqrt{\varepsilon^2+|\nabla u_\varepsilon|^2},
\]
where the potential
\[
V_\varepsilon (u)
:= 
\frac{1}{\varepsilon} (\phi -u ) _+  ^4
+\frac{1}{\varepsilon} (u-\psi)_+^4
\]
penalizes the violation of the constraint and a variant with exponent $2$ instead of $4$ was first proposed in the preprint~\cite{Mercier}. Here, for regularity reasons, we choose this higher power.

The main novelty of our work is a connection between this viscosity solution and distributional solutions in the spirit of~\cite{{MR1315658}} and~\cite{MR4810481}.
The latter solution concept is based on the gradient-flow structure of the mean curvature flow with obstacles instead of the comparison principle.
Also the vanishing viscosity approximation has a gradient-flow structure: Defining the energy 
\begin{equation}
    E_\varepsilon(\tilde u):=\int_\Omega \sqrt{\varepsilon^2+|\nabla \tilde u|^2} +\, V_\varepsilon(\tilde u)\, dx,
\end{equation}
we compute, with $\delta E_\varepsilon$ the first variation (or $L^2$ gradient) of the energy:
\begin{equation}\label{eq:EDI}
    \frac{d}{dt} E_\varepsilon(u_\varepsilon(\cdot,t)) 
    =\int_\Omega \delta E_\varepsilon (u_\varepsilon) \partial_t u_\varepsilon \,dx
    = -\int_\Omega (\delta E_\varepsilon(u_\varepsilon))^2 \sqrt{\varepsilon^2+|\nabla u_\varepsilon|^2} \,dx \leq0.
\end{equation}
In particular, given well-prepared initial conditions, we obtain
\[
    \int_0^\infty \int_\Omega (\delta E_\varepsilon(u_\varepsilon))^2 \sqrt{\varepsilon^2+|\nabla u_\varepsilon|^2} \,dx \, dt<\infty.
\]
The key estimate that we establish here is
\begin{equation}\label{eq:strange_estimate_intro}
    \sup_{t>0} \int_\Omega |\delta E_\varepsilon(u_\varepsilon)| \,dx <\infty.
\end{equation}
This estimate is of a different flavor than the gradient-flow estimate above, as it does not include the area factor $\sqrt{\varepsilon^2+|\nabla u_\varepsilon|^2}$ and thus does not correspond to an estimate on each individual level set in the limit $\varepsilon\to0$.

\begin{figure}
    \centering
    \includegraphics[width=0.45\linewidth]{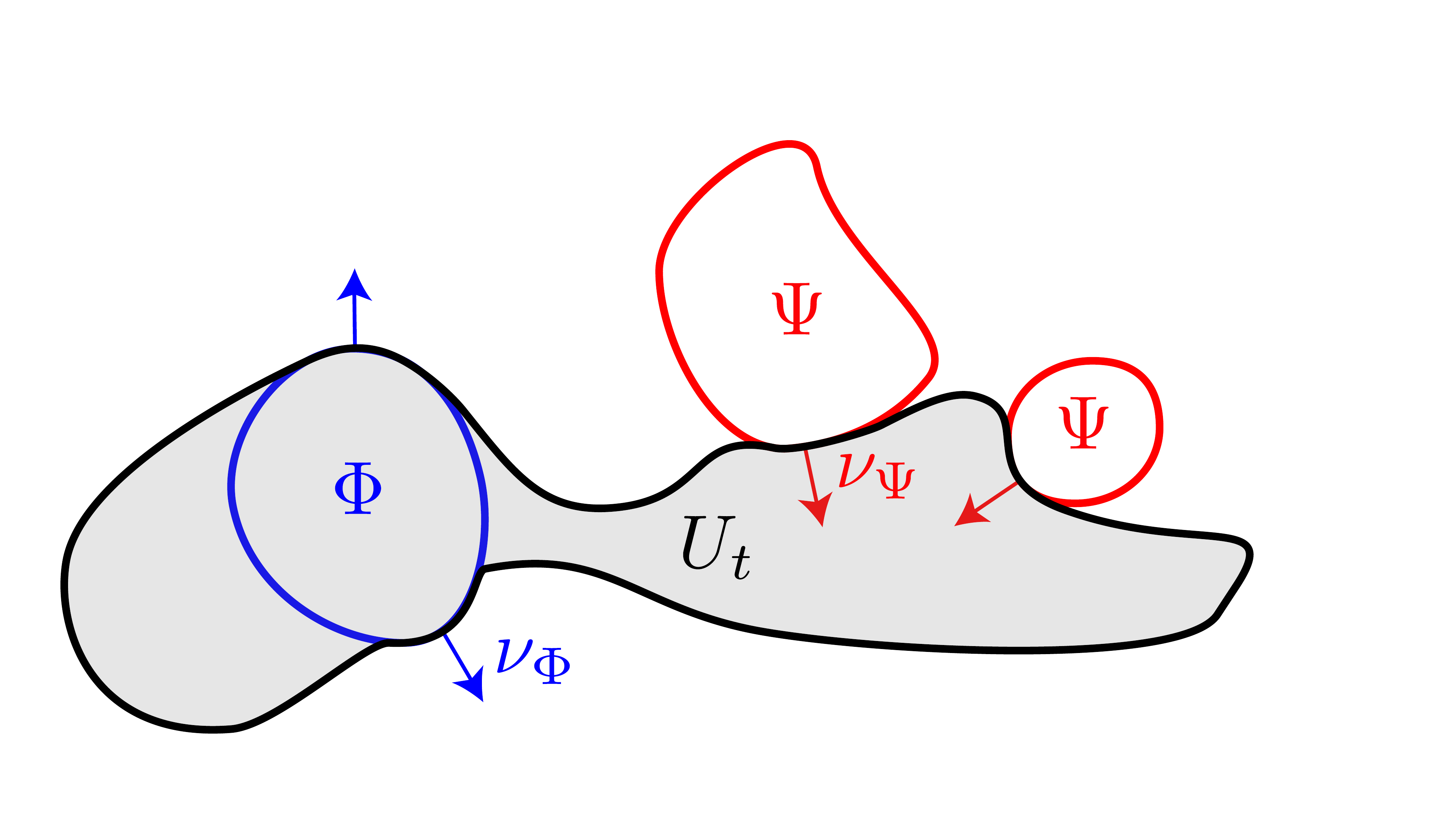}
    \includegraphics[width=0.45\linewidth]{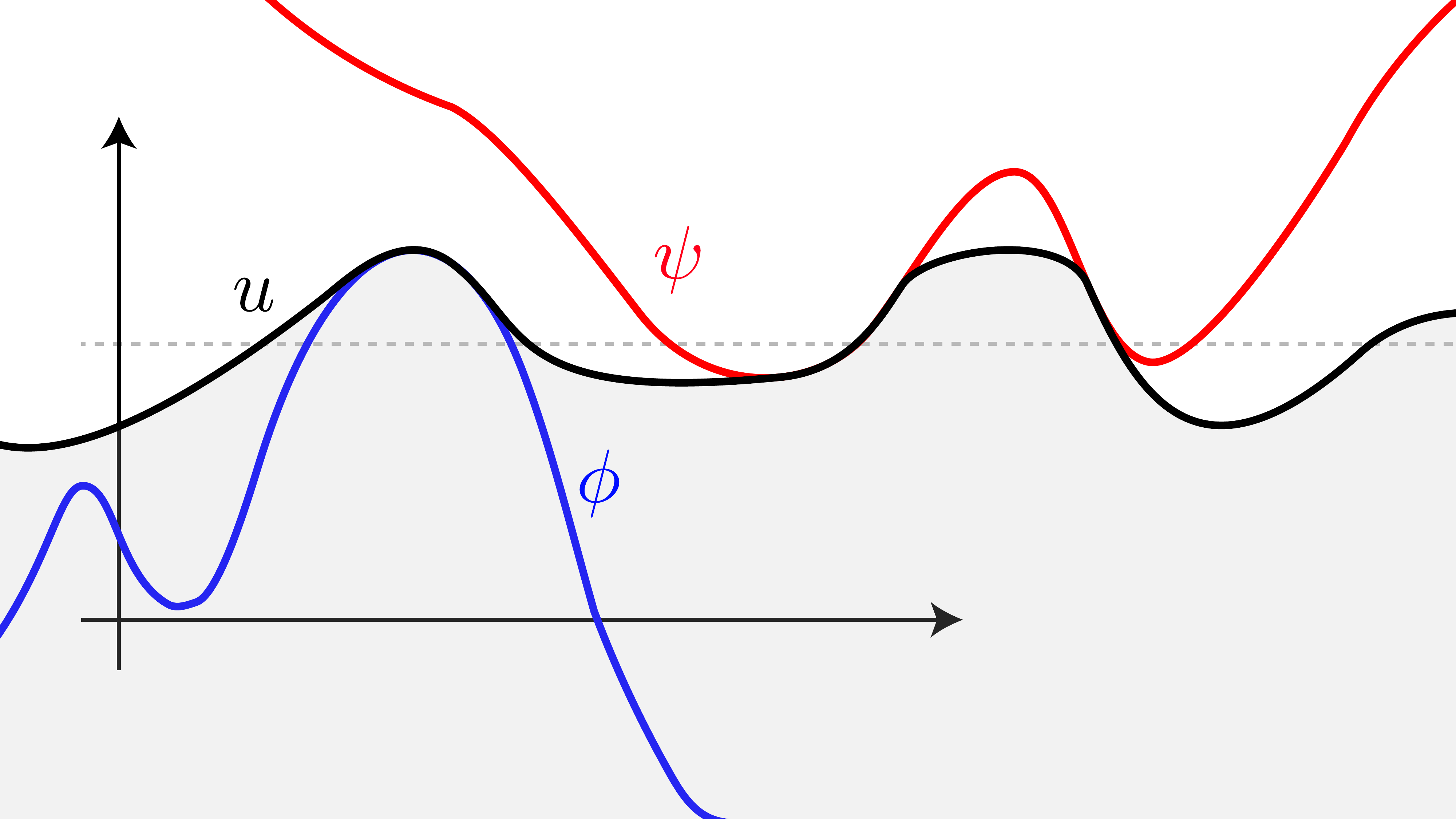}  
    \caption{A mean curvature flow with obstacles (left) and a corresponding level set flow (right; for simplicity, one spatial dimension is omitted here).}
    \label{fig:enter-label}
\end{figure}

We first verify the existence of viscosity solutions.
\begin{theorem}[Construction of viscosity solution]
\label{thm:viscositysol}
    Let $g \in C^2 (\Omega)$, $\phi, \psi\in C^1 (\Omega)$ with $\phi \leq g\leq \psi$ and let $u_\varepsilon$ be the unique classical solution of~\eqref{mcf}. 
    Then $u_\varepsilon \to u$ locally uniformly in $\Omega \times [0,\infty)$, where $u$ is the unique viscosity solution of~\eqref{eq:levelset_constraint}--\eqref{eq:levelset_contact} in the sense of Definition~\ref{def:viscositysol}. 
\end{theorem}
The proof can basically be found in the literature~\cite{Mercier,MR3746189}. We include it for completeness in Appendix~\ref{appendix}. 

\medskip

Now we can state our main result. It states that generic level sets of this viscosity solution are BV solutions to the obstacle problem.
In particular, this is the first derivation of (generic) existence of BV solutions in the presence of obstacles.

\begin{theorem}[Generic level sets are BV solutions]\label{thm:BVsol}
    Let $g, \phi, \psi$ be well-prepared data in the sense of Definition~\ref{def:well-prepared} below and let $u$ be the unique viscosity solution to~\eqref{eq:levelset_constraint}--\eqref{eq:levelset_contact}. 
    Then, for a.e.~$\gamma\in \R$, the evolving super-level set $U(t)=\{u(\cdot,t)>\gamma\}$ is a distributional solution to the mean curvature flow with obstacles $\Phi=\{\phi>\gamma\}$ and $\Psi = \{\psi>\gamma\}$ in the sense of Definition~\ref{def:BVsol} below.    
\end{theorem}

The proof of our main result, Theorem~\ref{thm:BVsol}, hinges on the key estimate~\eqref{eq:strange_estimate_intro}, which is the basis for a compensated compactness argument that allows us to prove (an over the level sets averaged version of) the distributional equation \[
    \int_0^\infty \int_\Omega u\partial_t \zeta \,dx\,dt= - \int_0^\infty \int_\Omega \zeta V |\nabla u| \,dx\,dt
\]
for any $\zeta\in C_c^1((0,\infty)\times \Omega)$ and
\[
    \int_0^\infty \int_\Omega  V \nu \cdot X |\nabla u| \,dx\,dt \geq -\int_0^\infty\int_\Omega (I-\nu \otimes \nu) \colon \nabla X|\nabla u| \,dx\,dt
\]
for any $X\in C_c^1((0,\infty)\times \Omega;\mathbb{R}^d)$  satisfying the constraints
\[
    X\cdot \nabla \phi \leq 0 \quad \text{on }\{u=\phi\}
    \quad \text{and} \quad 
    X\cdot \nabla \psi \geq 0 \quad \text{on }\{u=\psi\}.
\]
Then the statement for almost every level set follows from the relabeling property, and the co-area and layer-cake formulas.

\section{Main definitions and notation}
For simplicity, we assume that $\Omega =\mathbb{T}^d$, i.e., we impose periodic boundary conditions. 
In the following, we introduce the main notation that will be used throughout the paper.

 Our notion of BV solution is the following natural extension of the solution concept of Luckhaus and Sturzenhecker~\cite{MR1386964}.
\begin{definition}[BV solution to mean curvature flow with obstacles]\label{def:BVsol}
Let $\Phi, \Psi \subset \Omega$ be two open sets with $C^1$ boundaries such that $\Phi \cap \Psi = \emptyset$. 
A family of sets of finite perimeter $(U_t)_{t\in (0,T)}$ depending measurably on $t$ is called a distributional solution to mean curvature flow with obstacles $\Phi $ and $\Psi$ and initial conditions $U_0$ if the following conditions hold:

\begin{enumerate}[(i)]
\item \emph{Constraints.}
$\Phi \subset  U_t \subset \Psi^c$ for almost every $t \in (0,\infty)$,
\item \emph{Velocity.}
For $\partial^\ast U_t$, it holds that
\begin{equation}\label{BVsol:1}
\mathcal{H}^{d-1} (\partial^\ast U_0) <\infty,
\quad
\mathrm{ess\, sup}_{t \in (0,\infty)} 
\mathcal{H}^{d-1} (\partial^\ast U_t) <\infty \end{equation}
and there exists an $(\mathcal{H}^{d-1}\llcorner{\partial^\ast U_t})\otimes dt$-measurable function $V$ with
\begin{equation}\label{BVsol:2}
\int_0 ^{\infty }
\int_{\partial^\ast U_t} V^2 \, d\mathcal{H}^{d-1} dt<\infty
\end{equation}
such that for any $\zeta \in C^1 _c (\Omega \times [0,\infty))  $
we have
\begin{equation}\label{BVsol:3}
\int_0 ^{\infty} \int_{U_t}
\partial_t \zeta \, dxdt
= 
-\int_0 ^{\infty} \int_{\partial^\ast U_t} \zeta V  \, d\mathcal{H}^{d-1} dt
- \int_{U_0 } \zeta (\cdot,0) \, dx.
\end{equation}
\item \emph{Energy dissipation relation.} For a.e.\ $T' \in (0,\infty)$ it holds
\begin{equation}\label{BVsol:4}
\mathcal{H}^{d-1} (\partial^\ast U_{T'})
+
\int _{0} ^{T'} \int_{\partial^\ast U_t} V^2 \, d \mathcal{H}^{d-1} dt \leq \mathcal{H}^{d-1} (\partial^\ast U_0).
\end{equation}
\item \emph{Motion law.} We have 
\begin{equation}\label{BVsol:5}
\int_0 ^{\infty}
\int_{\partial^\ast U_t} 
\left\{
\left( I - \nu \otimes \nu \right) \colon
\nabla X
+ V\nu \cdot X
\right\}
\, d\mathcal{H}^{d-1}dt
\geq 0
\end{equation}
for any admissible test vector field $X \in C^1 (\Omega \times [0,\infty);\R^d)$, where $\nu =(\nu_1,\dots,\nu_d)$ denotes the measure theoretic (outer) unit normal of $U_t$.
Here, a vector field $X$ is called admissible if, 
\begin{equation}\label{BVsol:6}
	X \cdot \nu_\Phi \geq 0 \quad \text{on }
    \overline{M} \cap \partial \Phi\times[0,\infty)
	\quad
    \text{and}
	\quad
    X\cdot \nu_\Psi \geq 0 \quad \text{on } \overline{M} \cap \partial \Psi\times[0,\infty),
\end{equation}
where $M:=\cup_{t\in \{s\mid \partial^\ast U_s \, \text{exists}\}}\partial^\ast U_t \times \{t\} $,
$\nu_\Phi$ and $\nu_\Psi$ denote the exterior normal vectors to the obstacles $\Phi$ and $\Psi$, respectively.

    \end{enumerate}
\end{definition}

    Note that our weak solution concept in Definition~\ref{def:BVsol} encodes all relevant information of the evolution. Indeed, in case of the (expected) regularity $\partial U_t \in C^{1,1}$, we may integrate by parts in~\eqref{BVsol:5} to get 
    \[    
    \int_0^{\infty} \int_{\Gamma_t} ( H + V ) X \cdot 
    \nu\, d\mathcal{H}^{d-1}dt \geq 0
    \]
    for all test vector fields $X$ satisfying~\eqref{BVsol:6},
    which encodes $V =-H$ on $\Gamma_t\setminus (\partial \Phi \cup \partial \Psi)
    $, $V\geq -H$ on $\Gamma_t \cap \partial \Phi$, and $V \leq -H$ on $\Gamma_t \cap \partial \Psi$.

Next, we state the precise definition of well-preparedness for our initial data and obstacles.
\begin{definition}[Well-prepared data]\label{def:well-prepared}
    We call $\phi,\psi,g\colon \Omega \to \R$ \emph{well-prepared data} if $\phi,\psi \in C^1 (\Omega), g \in C^2 (\Omega)$,
    \[
        \phi<\psi 
        \quad \text{and} \quad
        \phi \leq g \leq \psi \quad \text{in }\Omega,
    \]
    and there exists $L, \ell<\infty$ such that
    \begin{equation}\label{lem3.4:assumption1}
        \max_{x \in \Omega} |g(x)| \leq L,
        \quad 
        \max_{x \in \Omega} \phi (x) \leq L,
        \quad 
        \text{and} 
        \quad
        \min_{x \in \Omega} \psi (x) \geq -L.
    \end{equation}
    and the sets of critical points
    \[
        \{ x \in \Omega \mid \phi (x) \geq - (L+l) \quad \text{and} \quad
    |\nabla \phi (x)|=0 \}
    \]
    and
    \[
    \{ x \in \Omega \mid \psi (x) \leq L+l \quad \text{and} \quad
    |\nabla \psi (x)|=0 \}
    \]
    are finite.
\end{definition}

\begin{remark}
The above assumptions for $\phi$ and $\psi$ are reasonable for our problem.
Let $U \subset (-\frac{1}{4},\frac{1}{4})^d$ be an open set. Assume that
there exists $C^1$-diffeomorphism $F : (-\frac{1}{2}, \frac{1}{2})^d \to (-\frac{1}{2}, \frac{1}{2})^d$ with 
$F (x)=x$ for any $x \in (-\frac{1}{2}, \frac{1}{2})^d\setminus (-\frac{1}{4}, \frac{1}{4})^d$
and $U=F(B_\frac18 (0))$. 
Then $U$ has $C^1$-boundary $\partial U$.
Let $\tilde \psi: [0, \infty) \to \R$
be a $C^1$ function such that $\tilde \psi (x)= \frac{128 L}{3} (s^2 - \frac{1}{64})$ for any $s \in [0,\frac14)$, $\tilde \psi (s)= 3L$ for any $s \geq \frac{3}{8}$, and $\tilde \psi' (s)>0$ for any $s \in (0,\frac38)$. Then the $C^1$ function $\psi (x) :=\tilde \psi (|F^{-1} (x)|)$ satisfies 
$\partial U =\{\psi=0\}$
and is well-prepared in the sense of Definition~\ref{def:well-prepared} with $l=L$, $\{ x \in \Omega \mid \psi (x) \leq L+l \ \text{and} \
    |\nabla \psi (x)|=0 \}=\{0\}$, by periodically extending the domain.
Indeed, if $\nabla |G(x)|^2=0$ with $G(x)=F^{-1} (x)$,
we have 
$G \frac{\partial G}{\partial x} =0$. 
Hence, if $\psi \leq 2L$, then the critical points above are uniquely determined, since $G$ is surjective and $\frac{\partial G}{\partial x}$ is a regular matrix.

Similarly, if $\{U^i _+ \}_{i=1} ^N$ and $\{ 
U^i _- \}_{i=1} ^M$ are  families of open sets with 
$\dist (U_\pm ^i,U_\pm ^j)>0$ if $i\not =j$, $\dist (U_+ ^i,U_- ^j)>0$ for any $i,j$, and there exists a family of
$C^1$-diffeomorphism $\{F_\pm ^i\}$
such that $F^i _\pm (x)=x$ for any $x \in (-\frac{1}{2}, \frac{1}{2})\setminus (-\frac{1}{4}, \frac{1}{4})$ and $U_\pm ^i = F_\pm ^i (B_\frac18 (0))$, one can find suitable $C^1$ functions $\phi$ and $\psi$ that satisfy $\{ \phi =0\}=\cup_{i=1} ^M \partial U_- ^i$, $\{ \psi =0\}=\cup_{i=1} ^N \partial U_+ ^i$, and are well-prepared in the sense of Definition~\ref{def:well-prepared}.
\end{remark}

For $\varepsilon >0$ and $p \in \R^d$, 
we define 
\begin{equation}
    |p|_\varepsilon = \sqrt{ \varepsilon ^2+|p|^2}.
\end{equation}
Let $g$ be a (smooth) function on $\Omega$ 
with 
\begin{equation}
    \phi (x) \leq g(x) \leq \psi (x) \qquad 
    \text{for all }  x \in \Omega
\end{equation}
and $u_\varepsilon$ be a classical solution of 
\begin{equation}
\left\{ 
\begin{array}{ll}
\partial_t u_\varepsilon 
=\left( \delta_{ij} - \dfrac{\partial_{x_i} u_\varepsilon \partial_{x_j} u_\varepsilon}{|\nabla u_\varepsilon |_\varepsilon^2} \right) \partial_{x_i x_j} u_\varepsilon
+ |\nabla u_\varepsilon|_\varepsilon f_\varepsilon ,& (x,t)\in \Omega \times (0, \infty),  \\
u_\varepsilon (x,0) = g (x) ,  &x\in \Omega,
\end{array} \right.
\label{mcf}
\end{equation}
where $f_\varepsilon = -V' _\varepsilon(u_\varepsilon)$. 
Here and throughout, we use Einstein's summation convention.

Set $C_0: =\sup_{\varepsilon \in (0,1)} E_\varepsilon (g)$. Then it holds that $C_0<\infty$ 
since $V_\varepsilon (g) \equiv 0$ and $\int _\Omega |\nabla g| \, dx <\infty$.
In addition, by~\eqref{eq:EDI}, we have
\begin{equation}
\label{eq:monotone}
E_\varepsilon (u_\varepsilon (\cdot,t)) \leq E_\varepsilon (g) \leq C_0
\quad
\text{for any} 
\ \varepsilon \in (0,1) \ \text{and} \ t\geq 0.
\end{equation}
Note that the solution $u_\varepsilon$ of \eqref{mcf} is the mean curvature flow of a graph in $(d+1)$ dimensions in a potential when $\varepsilon=1$ (see \cite{MR1025164} for the mean curvature flow of a graph without potential).

\bigskip

Set
\[
\nu_\varepsilon = 
(\nu_{\varepsilon,1}, \dots , \nu_{\varepsilon,d})
:=\frac{-\nabla u_\varepsilon }{ |\nabla u_\varepsilon|_\varepsilon }, \qquad
H_\varepsilon 
:=
\divergence \nu_\varepsilon
=
-\divergence\left( \frac{\nabla u_\varepsilon }{ |\nabla u_\varepsilon|_\varepsilon }\right)
=-\frac{1}{|\nabla u_\varepsilon|_\varepsilon}
a_{ij} ^\varepsilon \partial_{x_i x_j} u_\varepsilon, \quad 
f_\varepsilon := - V' _\varepsilon(u_\varepsilon),
\]
where the coefficients $a_{ij}^\varepsilon$ are the following entries of a positive semi-definite matrix
\[
a_{ij} ^\varepsilon 
:= 
\delta_{ij} - \nu_{\varepsilon,i} \nu_{\varepsilon,j}
=
\delta_{ij} - \frac{\partial_{x_i} u_\varepsilon \partial _{x_j} u_\varepsilon}{|\nabla u_\varepsilon|_\varepsilon ^2}
\]
for $i,j \in \{1,2,\dots, d\}$. 
With this notation, we can write~\eqref{mcf} in the more expressive form
\[
\frac{\partial_t u_\varepsilon}{|\nabla u_\varepsilon|_\varepsilon}
=-H_\varepsilon + f_\varepsilon.
\]


\section{Estimates, weak convergences}
The proofs of this section extend those in \cite{MR1315658} to the case of obstacle problems.
\subsection{$L^1$-estimate for the mean curvature with the forcing term}
We begin with the key estimate, which will allow us to carry out the desired compensated compactness argument.
\begin{theorem}\label{thm1.2}
    Assume that
    \begin{equation}
    \sup_{\varepsilon \in (0,1)}
    \int_{\Omega} |-H_\varepsilon (x,0) + f_\varepsilon (x,0)| \, dx <\infty.
    \end{equation}
    Then we have 
    \begin{equation}
    \sup_{\varepsilon \in (0,1),  t\geq 0}
    \int_{\Omega} |-H_\varepsilon (x,t) + f_\varepsilon (x,t)| \, dx <\infty.
    \end{equation}
\end{theorem}

\begin{proof}
We compute
\begin{equation}
-\partial _t \nu_{\varepsilon,i}
=\frac{1}{|\nabla u_\varepsilon|_\varepsilon}
a^\varepsilon _{ij}
\partial_{t, x_j} u_\varepsilon
=
\frac{1}{|\nabla u_\varepsilon|_\varepsilon}
a^\varepsilon _{ij}
\partial_{x_j} \{ |\nabla u_\varepsilon|_\varepsilon (-H_\varepsilon + f_\varepsilon) \}
\end{equation}
and
\begin{equation}
-\partial _t H_{\varepsilon}
=-\partial_{x_i} (\partial _t \nu_{\varepsilon,i})
=\partial_{x_i}
\left(
\frac{1}{|\nabla u_\varepsilon|_\varepsilon}
a^\varepsilon _{ij}
\partial_{x_j} \{ |\nabla u_\varepsilon|_\varepsilon (-H_\varepsilon + f_\varepsilon) \}
\right).
\end{equation}
Let $\eta :\R\to [0,\infty)$ be a smooth convex 
function with $\eta (0)=0$.
Then we see that
\begin{equation}
\begin{split}
    \frac{d}{dt}
    \int_{\Omega} \eta (-H_\varepsilon + f_\varepsilon) \, dx
    &= \int_{\Omega} \eta' (-H_\varepsilon + f_\varepsilon) (-\partial _t H_\varepsilon + \partial _t f_\varepsilon) \, dx \\
    &=  
    \int_{\Omega} \eta' (-H_\varepsilon + f_\varepsilon) (-\partial _t H_\varepsilon)  \, dx
    +
        \int_{\Omega} \eta' (-H_\varepsilon + f_\varepsilon) \partial_t f _\varepsilon \, dx.
\end{split}
\end{equation}
Since $\partial_t f_\varepsilon = -V''_\varepsilon (u_\varepsilon)\partial_t u_\varepsilon=-V''_\varepsilon (u_\varepsilon)(-H_\varepsilon+f_\varepsilon)|\nabla u_\varepsilon|_\varepsilon$ and $V''(s)
=\frac{12}{\varepsilon} (s-\psi )_+^2 
+\frac{12}{\varepsilon} (\phi -s )_+^2 \geq 0$ 
and, since convexity of $\eta$ and $\eta(0)=0$ imply $\eta' (s) s \geq 0$, we have $\int_{\Omega} \eta' (-H_\varepsilon + f_\varepsilon)\partial _t f_\varepsilon \, dx \leq 0.$
Thus
\begin{equation*}
\begin{split}
    \frac{d}{dt}
    \int_{\Omega} \eta (-H_\varepsilon + f_\varepsilon) \, dx
  	\leq &\,
    \int_{\Omega} \eta' (-H_\varepsilon + f_\varepsilon) 
    \partial_{x_i}
    \left(
    \frac{1}{|\nabla u_\varepsilon|_\varepsilon}
    a^\varepsilon _{ij}
    \partial_{x_j} \{ |\nabla u_\varepsilon|_\varepsilon (-H_\varepsilon + f_\varepsilon) \}
    \right) \, dx \\
    = & \,
    - \int_{\Omega} \eta'' (-H_\varepsilon + f_\varepsilon) 
    \partial_{x_i}
    (-H_\varepsilon + f_\varepsilon)
    \left(
    \frac{1}{|\nabla u_\varepsilon|_\varepsilon}
    a^\varepsilon _{ij}
    \partial_{x_j} \{ |\nabla u_\varepsilon|_\varepsilon (-H_\varepsilon + f_\varepsilon) \}
    \right) \, dx\\
    = &\, 
    - \int_{\Omega} \eta'' (-H_\varepsilon + f_\varepsilon) 
    a^\varepsilon _{ij} 
    \partial_{x_i}
    (-H_\varepsilon + f_\varepsilon)
    \partial_{x_j}
    (-H_\varepsilon + f_\varepsilon) \, dx \\
    &
    - \int_{\Omega} 
    \eta'' (-H_\varepsilon + f_\varepsilon) 
    \partial_{x_i}
    (-H_\varepsilon + f_\varepsilon)
    \frac{a^\varepsilon _{ij}}{|\nabla u_\varepsilon|_\varepsilon}
   (-H_\varepsilon + f_\varepsilon) 
   \partial_{x_j} |\nabla u_\varepsilon|_\varepsilon 
    \, dx.
\end{split}
\end{equation*}

The first right-hand side integral is non-positive, again due to the convexity of $\eta$---now in form of $\eta'' (s) \geq 0$---, and the fact that
$(a^\varepsilon _{ij})_{ij}$ is positive definite.
Integration in time yields
\begin{equation}\label{eq8}
\begin{split}
    & \int_{\Omega} \eta (-H_\varepsilon + f_\varepsilon) \, dx\Big|_{t=T} \\
    \leq & \,
    \int_{\Omega} \eta (-H_\varepsilon + f_\varepsilon) \, dx\Big|_{t=0}
    + C_\varepsilon \int _0 ^T \int_{\Omega} 
    \eta'' (-H_\varepsilon + f_\varepsilon) 
    |\nabla (-H_\varepsilon + f_\varepsilon)|
    |-H_\varepsilon + f_\varepsilon| 
    \, dx,
\end{split}
\end{equation}
where 
$
C_\varepsilon =
\frac{1}{\varepsilon} \sup_{x,t} \big| \nabla |\nabla u_\varepsilon|_\varepsilon \big|
$.
For any $\delta>0$, we employ a smooth convex function $\eta_\delta :\R \to [0,\infty)$ such that
$\eta_\delta (0)=0$, $\eta_\delta (s) \to |s|$ 
uniformly as $\delta \to 0$, $0\leq \eta'' _\delta(s) \leq \frac{C}{\delta}$, and
$\spt \eta'' _\delta \subset (-\delta,\delta)$.
Substituting $\eta=\eta_\delta$ into \eqref{eq8},
we have
\begin{equation}
\begin{split}
    &\int_{\Omega} \eta_\delta (-H_\varepsilon + f_\varepsilon) \, dx\Big|_{t=T} \\
    \leq & \, 
    \int_{\Omega} \eta_\delta (-H_\varepsilon + f_\varepsilon) \, dx\Big|_{t=0}
    + \frac{C_\varepsilon C}{\delta} \int _0 ^T \int_{\{ |-H_\varepsilon +f_\varepsilon| \leq \delta \}}  
    |\nabla (-H_\varepsilon + f_\varepsilon)|
    |-H_\varepsilon + f_\varepsilon| 
    \, dx \\
    \leq & \,
    \int_{\Omega} \eta_\delta (-H_\varepsilon + f_\varepsilon) \, dx\Big|_{t=0}
    + C_\varepsilon C
    \int _0 ^T \int_{\{ |-H_\varepsilon +f_\varepsilon| \leq \delta \}}  
    |\nabla (-H_\varepsilon + f_\varepsilon)|
    \, dx.
\end{split}
\end{equation}
Taking the limit $\delta\downarrow0$, since $|\nabla (-H_\varepsilon + f_\varepsilon)| =0$
a.e.\ on $\{ |-H_\varepsilon +f_\varepsilon| =0 \}$, we obtain
\begin{equation}
\begin{split}
    \int_{\Omega} | -H_\varepsilon + f_\varepsilon| \, dx\Big|_{t=T}
    \leq  \, 
    \int_{\Omega} | -H_\varepsilon + f_\varepsilon| \, dx\Big|_{t=0},
\end{split}
\end{equation}
which is precisely our claim.
\end{proof}
For $u \in W ^{1,\infty} (\Omega)$, we define the limit energy
\[
E(u) :=
\begin{cases}
\displaystyle \int_{\Omega} |\nabla u| \, dx, &
\text{if} \ \ \{ x \in \Omega \mid u(x) > \psi (x) \ \text{or} \ u(x) < \phi (x) \} =\emptyset, \\
\qquad +\infty, & \text{otherwise.}    
\end{cases}
\]
Next, we show the simple fact that the approximate energies $E_\varepsilon$ converge to $E$.

\begin{theorem}\label{thm1.3}
Let $\{ \varepsilon _i \} _{i=1} ^\infty$ be a positive sequence with $\varepsilon _i \to 0$ as $i\to \infty$. Then
$E_{\varepsilon_i}$ $\Gamma$-converges to $E$ in the uniform topology.  
More precisely, the following hold:
\begin{enumerate}[(i)]
    \item \label{Gamma_(i)}  If 
    $u_{i} \to u $ (locally) uniformly on $ \Omega $ 
    as $ i\to \infty$ for
    $\{ u_{i} \}_{i=1} ^\infty \subset W^{1,\infty} (\Omega)$ and $u \in W^{1,\infty} (\Omega)$, then
    \[
    E(u) \leq \liminf_{i\to \infty} E_{\varepsilon_i} (u_i)
    \]
    \item \label{Gamma_(ii)} It holds that
    \[
    \limsup_{i \to \infty} E_{\varepsilon_i} (u)
    \leq E (u) \qquad \text{for any} \ u \in 
    W^{1,\infty}  (\Omega).
    \]
\end{enumerate}
\end{theorem}
\begin{proof}
    \emph{Argument for~\eqref{Gamma_(i)}:}  Assume that $u_{i} \to u $ (locally) uniformly on $ \Omega $ 
    as $ i\to \infty$ for
    $\{ u_{i} \}_{i=1} ^\infty \subset W^{1,\infty} (\Omega)$ and $u \in W^{1,\infty} (\Omega)$.
    By the lower semi-continuity of BV functions, we have
    \[
    \int _{\Omega} |\nabla u| \, dx \leq 
    \liminf_{i\to \infty} \int_{\Omega} |\nabla u_i| \, dx 
    \leq 
    \liminf_{i\to \infty} \int_{\Omega} |\nabla u_i|_{\varepsilon _i} \, dx.
    \]    
    Without loss of generality, we may assume $\sup_{i \in \mathbb{N}} \int_{\Omega} V_{\varepsilon_i} (u_i) \, dx <\infty$. Then
    one can easily check that
    \[
    \{ x \in \Omega \mid u(x) > \psi (x) \ \text{or} \ u(x) < \phi (x) \}=\emptyset.
    \]
    Therefore $E (u) =\int_{\Omega} |\nabla u| \, dx$ and
    $E(u) \leq \liminf_{i\to \infty} E_{\varepsilon_i} (u_i)$.
    Similarly, if $u\in W^{1,\infty}(\Omega)$ with $E(u)=\infty$, then necessarily $\lim_{i\to \infty} E_{\varepsilon_i} (u_i)= \infty$.
    
    \smallskip
	
	\emph{Argument for~\eqref{Gamma_(ii)}:} For any $u \in W^{1,\infty} (\Omega)$,
    we may assume that $E(u) <\infty$.
    Then $\phi (x) \leq u (x) \leq \psi (x)$ for any $x \in \Omega$. Therefore $V_{\varepsilon_i} (u (x))=0$ for any $x \in \Omega$ and
    \[
    \limsup_{i\to\infty} E_{\varepsilon_i} (u)
    =
    \limsup_{i\to\infty}
    \int_\Omega |\nabla u|_{\varepsilon_i} + V_{\varepsilon_i} (u) \, dx
    \leq
    \limsup_{i\to\infty}
    \int_\Omega |\nabla u| + \varepsilon \, dx
    =E(u),
    \]
    where we used $|\Omega|=1$.
\end{proof}

The following proposition states the basic a priori estimates for our approximations $u_\varepsilon$. We define $\Omega _T=\Omega \times [0,T]$ for any $T>0$.
\begin{proposition}\label{prop:mp}
For any $T>0$ we have the following a priori estimates.
\begin{enumerate}[(i)]
    \item \label{item:mp1}
    Uniform bounds:
    \[
    \min \{  
    \min_{x \in \Omega} \psi (x), \min_{x \in \Omega } g (x) \}
    \leq
    \min _{(x,t) \in \Omega_T } u_\varepsilon (x,t) \leq 
    \max _{(x,t) \in \Omega_T } u_\varepsilon (x,t)
    \leq \max \{  
    \max_{x \in \Omega} \phi (x), \max_{x \in \Omega } g (x) \}.
    \]
    \item\label{item:mp2}
    Gradient bound:
    \[
    \max _{(x,t) \in \Omega_T } 
    | \nabla u_\varepsilon (x,t)|
    \leq \max \{  
    \max_{x \in \Omega} |\nabla \phi (x)|,
    \max_{x \in \Omega} |\nabla \psi (x)|,    
    \max_{x \in \Omega } |\nabla g (x)| \}.
    \]
    \item\label{item:mp3}
    Bound on time derivative:
    \[
    \max _{(x,t) \in \Omega_T } 
    | \partial_t u_\varepsilon (x,t)|
    \leq
    \max_{x \in \Omega } |\nabla ^2 g (x)|.
    \]
\end{enumerate}

\end{proposition}

\begin{proof}
  \emph{Argument for~\eqref{item:mp1}.} We define the excess sets
\[
\Omega _T ^+ :=\{(x,t) \in \Omega_T \mid
 u(x,t) > \phi (x) \},
 \qquad 
 \Gamma ^+ := \partial \Omega_ T ^+ \setminus 
 \{ (x,t) \in \Omega \times [0,\infty) \mid t=T \}.
\]
Then $u_\varepsilon$ is the sub-solution of the standard level set flow equation on $\Omega_ T ^+$, that is,
\[
\frac{\partial _ t u_\varepsilon}{|\nabla u_\varepsilon|_\varepsilon} \leq
\divergence \left( \frac{\nabla u_\varepsilon}{|\nabla u_\varepsilon|_\varepsilon} \right)
\qquad \text{on} \ \Omega _T ^+.
\]
Then the standard maximum principle implies
\begin{equation}\label{mp}
    \max _{(x,t) \in \overline{\Omega_T ^+}} u_\varepsilon (x,t)
=
\max_{(x,t) \in \Gamma ^+} u_\varepsilon (x,t)
\leq \max \{ \max_{x \in \Omega} \phi (x), \max_{x \in \Omega } g (x) \}.
\end{equation}
In addition, $u_\varepsilon (x,t) \leq \phi (x)$ for any $(x,t) \not \in \Omega _T ^+$. By this and \eqref{mp} we obtain the second inequality.
The rest of the proof can be shown in the same way.

\bigskip

\emph{Argument for~\eqref{item:mp2}.} Let $k \in \{1,2,\dots,d\}$ and
$\Omega _T ^k := \Omega _T \cap \{(x,t) \mid \partial_{x_k} u_\varepsilon (x,t) > \max\{ |\partial _{x_k} \psi (x)|, |\partial _{x_k} \phi (x)| \} \}$.
We compute
\[
\partial_{x_k} f_\varepsilon 
= \frac{12}{\varepsilon}
\{
- ((u-\psi)_+) ^2 \partial _{x_k}( u_\varepsilon - \psi )
+
((\phi - u_\varepsilon)_+) ^2 \partial _{x_k}( \phi -u_\varepsilon)
\}.
\]
Note that $\partial_{x_k} f_\varepsilon  \leq 0$ on $\Omega_ T ^k$.
Therefore we have
\begin{equation}\label{mp2}
\begin{split}
\partial_t \partial_{x_k} u_\varepsilon
= & \,
\partial _{x_k} \left( 
|\nabla u_\varepsilon|_\varepsilon \divergence \left( \frac{\nabla u_\varepsilon}{|\nabla u_\varepsilon|_\varepsilon} \right)
\right)
+ \partial_{x_k} (|\nabla u_\varepsilon |_\varepsilon f_\varepsilon )\\
\leq & \,  
\partial _{x_k} \left( 
|\nabla u_\varepsilon|_\varepsilon
\divergence \left( \frac{\nabla u_\varepsilon}{|\nabla u_\varepsilon|_\varepsilon} \right)
\right)
+  
f_\varepsilon
\frac{\nabla u_\varepsilon \cdot \nabla 
(\partial_{x_k} u_\varepsilon) }{|\nabla u_\varepsilon |_\varepsilon} 
\end{split}
\end{equation}
holds on $\Omega_T ^k$. 
We compute
\begin{equation}\label{mp3}
    \begin{split}
    &\partial _{x_k} \left( 
    |\nabla u_\varepsilon|_\varepsilon    \divergence \left( \frac{\nabla u_\varepsilon}{|\nabla u_\varepsilon|_\varepsilon} \right)
    \right)
    =
    \partial _{x_k}
    \left\{
    \left(
    \delta_{ij} - \dfrac{\partial_{x_i} u_\varepsilon \partial_{x_j} u_\varepsilon}{|\nabla u_\varepsilon |^2 +\varepsilon ^2}
    \right)
    \partial_{x_i x_j} u_\varepsilon
    \right\}\\
    = & \, 
    a_{ij} ^\varepsilon \partial_{x_i x_j } (\partial_{x_k} u_\varepsilon)
    -\frac{2 \partial_{x_i x_j} u_\varepsilon}{(|\nabla u_\varepsilon|^2 +\varepsilon^2)^2}
    \left\{
    (|\nabla u_\varepsilon|^2 +\varepsilon^2)
    \partial_{x_i} u_\varepsilon \partial_{x_j} (\partial_{x_k} u_\varepsilon)
    - \partial_{x_i} u_\varepsilon \partial_{x_j} u_\varepsilon \nabla u_\varepsilon \cdot 
    \nabla (\partial_{x_k} u_\varepsilon)
    \right\}.
    \end{split}
\end{equation}
By \eqref{mp2} and \eqref{mp3}, 
\[
\partial_t (\partial_{x_k} u_\varepsilon)
\leq 
a_{ij} ^\varepsilon \partial_{x_i x_j } (\partial_{x_k} u_\varepsilon)
+ B \cdot \nabla (\partial_{x_k} u_\varepsilon)
\qquad \text{on} \ \Omega_T ^k,
\]
for some vector-field $B$.
By the definition of $\Omega_T ^k$,
$ \partial_{x_k} u_\varepsilon (x,t)
\leq \max_{x \in \Omega}\{ |\partial _{x_k} \psi |, |\partial _{x_k} \phi |, |\partial _{x_k} g | \}$ on the parabolic boundary of $\Omega _T ^k$.
Hence $\partial _{x_k } u_\varepsilon (x,t)
\leq \max_{ \Omega}\{ |\partial _{x_k} \psi|, |\partial _{x_k} \phi|, |\partial _{x_k} g| \}$ on $\overline{\Omega_T ^k}$, by the maximum principle. Note that this inequality also holds on $\overline{\Omega_T }$
since $\partial _{x_k} u_\varepsilon (x,t) \leq \max_{x \in \Omega} \{ |\partial_{x_k} \psi |, |\partial_{x_k} \phi |\}$ if $(x,t) \not \in \Omega_T ^k$.
Repeating the same argument, we obtain~\eqref{item:mp2}.
Since $\partial_t \phi =\partial _t \psi \equiv 0$, we obtain~\eqref{item:mp3} similarly.
\end{proof}

For the convenience of the reader we briefly check the convexity of our integrand.
\begin{lemma}\label{lemma:convexity}
Set $L(p,z,x) = \sqrt{|p|^2 +\varepsilon^2} + 
\frac{1}{\varepsilon} ((z-\psi (x))_+)^4
+ 
\frac{1}{\varepsilon} ((\phi (x) -z)_+)^4
$ for $p \in \R^d$ and $x,z \in \R$.
Then the mapping $(p,z) \mapsto L(p,z,x)$ is convex, that is,
\[
L(p,z,x) + \nabla _p L (p,z,x) \cdot (q-p)
+ \partial_z L (p,z,x) (w-z) \leq L (q,w,x).
\]
\end{lemma}
\begin{proof}
    $D^2 _p L= \frac{1}{\sqrt{|p|^2 +\varepsilon^2}}( \delta_{ij} - q_i q_j )_{ij} $, where $q=(q_1,\dots,q_d)=\frac{p}{\sqrt{|p|^2+\varepsilon^2}}$.
    Thus, $(D^2 _p L) z = \frac{1}{\sqrt{|p|^2 +\varepsilon^2}} (z - (z\cdot q)q)$ for any $z \in \R^d$. Hence, any eigenvector has a positive eigenvalue. In addition, 
    $\partial _z ^2 L= \frac{12}{\varepsilon}
    \{ ((z-\psi (x))_+)^2+ 
    ((\phi (x) -z)_+)^2 \geq 0
    $ and $\partial_{p_i, z} L= 0$, any eigenvalue of $D^2 _{p,z} L$ is nonnegative.
    Hence, the mapping is convex.
\end{proof}

The following crucial result strengthens the convergence of our approximation of the levelset solution and states that their energy converges to the energy of the limit. In other words, for each time, our approximations form a recovery sequence for their limit in the $\Gamma$-convergence.
The result hinges on the central estimate from Theorem~\ref{thm1.2}, and on the convexity of our energy from Lemma~\ref{lemma:convexity}.

\begin{theorem}\label{thm:2.5}
Let $u_\varepsilon$ be the solution of~\eqref{mcf} and $u$ be the viscosity solution of~\eqref{eq:levelset_constraint}--
\eqref{eq:levelset_contact}. Then there exist a subsequence $\varepsilon_i \to 0$
(denoted by same index) and $u \in W^{1,\infty} (\Omega \times [0,T])$ such that
\[
\nabla u_{\varepsilon_i} \rightharpoonup \nabla u \qquad \text{weakly-}\ast \quad \text{in} \ L^\infty (\Omega \times (0,T);\R^d),
\]
\[
u_{\varepsilon _i} \to u \quad \text{uniformly on} \ \Omega\times [0,T], \quad
\lim_{i\to \infty} E_{\varepsilon_i} (u_{\varepsilon_i}) =E (u), \quad \forall\, t\geq 0,
\]
and
\[
\phi (x) \leq u(x,t) \leq \psi (x) \qquad \text{for any} \ (x,t) \in \Omega \times [0,T].
\]
\end{theorem}
\begin{proof}
By Proposition \ref{prop:mp}, there exists $C>0$ such that
\[
\sup_{\varepsilon \in (0,1)} 
\| u_\varepsilon,\nabla u_\varepsilon,\partial _t u_\varepsilon \|_\infty \leq C
\]
holds. Hence there exist a subsequence $\varepsilon_i \to 0$
(denoted by same index) and $u \in W^{1,\infty} (\Omega)$ such that
\begin{equation}\label{eq:AA}
\nabla u_{\varepsilon_i} \rightharpoonup \nabla u \quad \text{weakly-}\ast \quad \text{in} \ L^\infty (\Omega \times (0,T);\R^d) \quad \text{and} \quad
    u_{\varepsilon _i} \to u \quad \text{uniformly on} \ \Omega\times [0,T]
\end{equation}
by the Arzel\`{a}-Ascoli theorem and the standard theory of the weak-$\ast$ convergences.
We only need to show 
$\phi \leq u\leq \psi $ and $\lim_{i\to \infty} E_{\varepsilon_i} (u_{\varepsilon_i}) =E (u)$ for this subsequence.

From the convexity, we have
\begin{equation}\label{eq:13}
\begin{split}
E_\varepsilon (\eta)
\geq & \,
E_\varepsilon (u_\varepsilon)
+ \int_{\Omega} \frac{\nabla u_\varepsilon}{|\nabla u_\varepsilon|_\varepsilon }
\cdot (\nabla \eta -\nabla u_\varepsilon)
+V' _\varepsilon (u_\varepsilon) (\eta -u_\varepsilon) \, dx \\
=& \,
E_\varepsilon (u_\varepsilon)
- \int_{\Omega} (-H_\varepsilon +f_\varepsilon ) (\eta -u_\varepsilon) \, dx
\end{split}
\end{equation}
for the solution $u_\varepsilon$ and for any $\eta \in C^\infty (\Omega)$.
Note that we may use $C^\infty (\Omega)$ as the space of test function because $\Omega =\mathbb{T}^d$.
Fix $t\geq 0$ and define a signed measure $\mu_\varepsilon ^t$
by $d\mu_\varepsilon ^t := (-H_\varepsilon +f_\varepsilon) \, dx$. Then $\sup_{\varepsilon \in (0,1)} |\mu_\varepsilon ^t|(\Omega) <\infty$ by Theorem \ref{thm1.2}.
Therefore there exist a signed measure $\mu ^t$
and a subsequence $\varepsilon_i ^t \to 0$ 
which depends on $t$ such that 
\[
\int \eta \, d\mu_{\varepsilon_i ^t} ^t 
\to \int \eta \, d\mu ^t
\qquad 
\text{for any} \ \eta \in C(\Omega).
\]
Note that we may use $C(\Omega)$ 
as a function space for test functions since this space is separable.
Since $u_\varepsilon (\cdot,t)$ converges to $u (\cdot,t)$ on $\Omega$ uniformly and 
$\sup_{\varepsilon \in (0,1)} |\mu_\varepsilon ^t|(\Omega) <\infty$, 
we have
\[
\int_{\Omega} (-H_{\varepsilon_i ^t} +f_{\varepsilon_i ^t} ) (\eta -u_{\varepsilon_i ^t}) \, dx
=
\int_{\Omega} (\eta -u) \, d\mu_{\varepsilon_i ^t} ^t
+
\int_{\Omega} (u-u_{\varepsilon_i ^t}) \, d\mu_{\varepsilon_i ^t} ^t
\to 
\int_{\Omega} (\eta -u) \, d\mu ^t
\]
as $i \to \infty$. From this, \eqref{eq:13}, and Theorem \ref{thm1.3}, we obtain
\[
E (\eta) \geq \limsup_{i\to \infty} E_{\varepsilon_i ^t } (\eta)
\geq \limsup_{i\to \infty} E_{\varepsilon_i ^t} (u_{\varepsilon_i ^t})
-\int_\Omega (\eta -u) \, d\mu^t
\]
for any $\eta \in C^\infty (\Omega)$.
Letting $\eta \to u$ in $W^{1,\infty}(\Omega)$, 
we get 
\[
E(u) \geq \limsup_{i\to \infty} E_{\varepsilon_i ^t} (u_{\varepsilon_i ^t}).
\]
Using Theorem \ref{thm1.3} again, we obtain
$\limsup_{i\to \infty} E_{\varepsilon_i ^t} (u_{\varepsilon_i ^t})\leq E(u) \leq \liminf_{i\to \infty} E_{\varepsilon_i ^t} (u_{\varepsilon_i ^t})$. Hence
$\lim_{i\to \infty} E_{\varepsilon_i ^t} (u_{\varepsilon_i ^t}) =E (u)$ for some subsequence $\{\varepsilon_i ^t\}$.
By repeating the same arguments, any subsequence $\{ \varepsilon _{i_j} \} \subset \{ \varepsilon_i \}$ satisfies
$\limsup_{j\to \infty} E_{\varepsilon_{i_j} } (u_{\varepsilon_{i_j}})\leq E(u) \leq \liminf_{j\to \infty} E_{\varepsilon_{i_j}} (u_{\varepsilon_{i_j}})$. 
This leads to the conclusion $\lim_{i\to \infty} E_{\varepsilon_i} (u_{\varepsilon_i}) =E (u)$ for any $t\geq 0$.
By \eqref{eq:monotone}, $E(u)\leq C_0 <\infty$ for any $t\geq 0$, and hence we have $\phi \leq u\leq \psi$ on $\Omega \times [0,T]$.
\end{proof}

\subsection{Convergence of the unit normal vectors}
In this subsection, we suppress the dependence on time and assume that $u_\varepsilon \in C^1(\Omega)$ and $ u \in W^{1,\infty} (\Omega)$ with $\Omega=\mathbb{T}^d$ satisfy
the following:
\begin{equation}
\label{A1}
\tag{A1}
\sup_{\varepsilon \in (0,1), x \in \Omega} \{ |u_\varepsilon (x)|, |\nabla u_\varepsilon (x)| \}<\infty.
\end{equation}
\begin{equation}
\label{A2} 
\tag{A2}
u_\varepsilon \to u \ \text{uniformly in} \ \Omega \ \ \text{and} \ \ \phi \leq u \leq \psi \  \text{in} \ \Omega.
\end{equation}
\begin{equation}
\label{A3} 
\tag{A3}
\nabla u_\varepsilon \rightharpoonup \nabla u \quad \text{weakly-}\ast \ \text{in} \ L^\infty (\Omega ; \mathbb{R}^d).
\end{equation}
\begin{equation}
\label{A4}
\tag{A4}
\sup_{\varepsilon \in (0,1)} \int_{\Omega} |-H_\varepsilon + f_\varepsilon| \, dx<\infty,
\end{equation}
where $H_\varepsilon = \divergence \nu_\varepsilon = -\divergence \left( \frac{\nabla u_\varepsilon}{|\nabla u_\varepsilon |_\varepsilon} \right)$ and $f_\varepsilon=-V_\varepsilon'(u_\varepsilon)$.

We define the contact sets $\Omega _+ :=\{ x\in \Omega \mid u(x)=\psi (x) \}$ and
$\Omega _- :=\{ x\in \Omega \mid u(x)=\phi (x) \}$. Since $u$ is uniformly Lipschitz and $\inf_{x \in \Omega} \{\psi (x)-\phi (x)\}>0$,
we have $\dist (\Omega_+,\Omega_-)>0$.
Define $\Omega _\pm ^\delta := \{ x\in \Omega \mid \dist(x,\Omega_\pm) < \delta \}$ for $\delta>0$. Then one can easily check that
$\dist (\Omega_+ ^\delta, \Omega_- ^\delta)>0$
for sufficiently small $\delta>0$, and
\begin{equation}\label{omega-delta}
    \inf_{x \in (\Omega _+ ^\delta)^c} \{\psi (x) -u(x)\} >0
\qquad
\text{and}
\qquad
\inf_{x \in (\Omega _- ^\delta)^c} \{u (x) -\phi (x)\} >0.
\end{equation}
\begin{lemma}\label{lemma2.6}
Let $\delta>0$
satisfy $\dist (\Omega_+ ^\delta, \Omega_- ^\delta)>0$. 
Then for sufficiently small $\varepsilon>0$ we have
$\{ x \in \Omega \mid f_{\varepsilon} (x) <0 \} \subset \Omega_+ ^\delta$ and
$\{ x \in \Omega \mid f_{\varepsilon} (x) >0 \} \subset \Omega_- ^\delta$.
\end{lemma}
\begin{proof}
We may assume $\dist (\Omega_+ ^\delta, \Omega_- ^\delta)>0$. Since $u_\varepsilon$ converges to $u$ on $\Omega$ uniformly, for sufficiently small $\varepsilon>0$, we have
\begin{equation}\label{omega-delta2}
    \inf_{x \in (\Omega _+ ^\delta)^c} \{\psi (x) -u_\varepsilon(x)\} >0
\qquad
\text{and}
\qquad
\inf_{x \in (\Omega _- ^\delta)^c} \{u_\varepsilon (x) -\phi (x)\} >0,
\end{equation}
where we used \eqref{omega-delta}.
By \eqref{omega-delta2} and
\[
f_{\varepsilon} 
= \frac{4}{\varepsilon} (\phi-u_{\varepsilon})_+^3-\frac{4}{\varepsilon} (u_{\varepsilon}-\psi)_+^3,
\]
if $x \in (\Omega _+ ^\delta)^c $, 
$f_{\varepsilon} 
= \frac{4}{\varepsilon} (\phi-u_{\varepsilon})_+^3 \geq 0$.
Hence $\{ x \in \Omega \mid f_{\varepsilon} (x) <0 \} \subset \Omega_+ ^\delta$.
The rest of the proof can be shown in the same way.
\end{proof}

The main result of this section and the main ingredient for Theorem~\ref{thm:BVsol} is the following strict convergence in $BV$. This follows from a compensated compactness argument based on the key estimate~\eqref{A4}, which in turn follows from~Theorem~\ref{thm1.2}.
\begin{theorem}\label{thm:2.7}
Under the assumptions~\eqref{A1}--\eqref{A4}, $| \nabla u_\varepsilon | \rightharpoonup |\nabla u|$ weakly-$\ast$ in $L^\infty (\Omega )$.
\end{theorem}

\begin{proof}
By the boundedness of $\nabla u_\varepsilon$, there exists a subsequence $\varepsilon _i \to 0$ and a Borel measure $\nu_x$ for a.e.\ $x \in \Omega$ such that
\begin{equation}\label{young1}
\nu_x (\Omega)=1, \qquad F (\nabla u_{\varepsilon _i}) \rightharpoonup \overline{F} \ 
\text{weakly-}\ast \ \text{in} \ L^\infty (\Omega)
\end{equation}
for any $F \in C(\mathbb{R}^d ; \R)$, where
\begin{equation}\label{young2}
\overline{F} (x):= \int_{\Omega} F(\lambda) \,d\nu _x (\lambda) \qquad \text{for a.e.} \ x \in \Omega
\end{equation}
(see \cite[Theorem 11]{MR1034481}).
Since $|\nu _\varepsilon (x)| \leq 1$ for any $x \in \Omega$, by taking the subsequence, 
we may assume that there exists $\nu \in L^\infty (\Omega; \R^d)$ such that $\nu_{\varepsilon_i} \rightharpoonup \nu$ weakly-$\ast$ in $L^\infty (\Omega; \R^d)$. By using
\[
\| \nu \|_{L^\infty (\Omega; \R^d)} = \sup_{\eta \in L^1 (\Omega; \R^d), \ \| \eta \|_{L^1}=1} 
\left| \int_{\Omega} \nu \cdot \eta \, dx \right|,
\]
we have $|\nu (x)| \leq 1$ for a.e.\ $x \in \Omega$.
From the boundedness of $|- H_\varepsilon +f_\varepsilon|$, by taking the subsequence if necessary, we have
\[
\int_{\Omega} (-H_{\varepsilon_i}  + f_{\varepsilon_i} ) \eta \, dx 
\to \int_{\Omega} \eta \, d\mu,
\qquad \forall \, \eta \in C(\Omega)
\]
for some signed measure $\mu$.
Then we compute that
\begin{equation}\label{eq:15}
\begin{split}
&\lim_{i\to \infty} \int_\Omega f_{\varepsilon _i}  \eta \, dx
= \lim_{i\to \infty} \int_{\Omega} (-H_{\varepsilon_i} + f_{\varepsilon _i} ) \eta
+ H_{\varepsilon _i} \eta \, dx \\
=&  \int_{\Omega} \eta \, d\mu - \lim_{i\to \infty} \int_\Omega \nu _{\varepsilon_i} \cdot \nabla \eta \, dx
=\int_{\Omega} \eta \, d\mu 
- \int_\Omega \nu \cdot \nabla \eta \, dx,\qquad \forall \, \eta \in W^{1,\infty} (\Omega).
\end{split}
\end{equation}
Thus the limit of $\int_\Omega f_{\varepsilon _i} \eta \, dx$ exists.
By Lemma \ref{lemma2.6}, we may choose $\eta \in C^1 (\Omega)$ such that $|\eta|\leq 1$, $\eta = -1$ on $\Omega_+ ^\delta$
and $\eta = 1$ on $\Omega_- ^\delta$. Substituting this into \eqref{eq:15}, we obtain
\[
\limsup_{i\to \infty} \int_{\Omega} |f_{\varepsilon_i}|\, dx<\infty.
\]
By using this and $\lim_{i\to \infty}\|u_{\varepsilon_i} -u\|_\infty =0$, 
we have
\begin{equation}\label{eq:16}
\left| \int_{\Omega} f_{\varepsilon _i}  (u_{\varepsilon_i } -u) \eta \, dx \right|
\leq \| u_{\varepsilon_i} -u \|_\infty  \| \eta \|_\infty
\int_{\Omega} |f_{\varepsilon _i}  | \, dx \to 0 \qquad 
\text{as} \ i \to \infty,\qquad \forall \, \eta \in W^{1,\infty} (\Omega).
\end{equation}
By \eqref{eq:15} and \eqref{eq:16}, we obtain
\begin{equation}\label{eq:17}
\begin{split}
\lim_{i\to \infty} \int_{\Omega}
f_{\varepsilon_i}  u_{\varepsilon_i} \eta \, dx
=& \,
\lim_{i\to \infty}
\int_{\Omega}
f_{\varepsilon_i}  (u_{\varepsilon_i}-u) \eta \, dx
+
\lim_{i\to \infty}
\int_{\Omega}
f_{\varepsilon_i}  u \eta \, dx \\
=& \,
\int_{\Omega} u\eta \, d\mu
- \int_{\Omega} \nu \cdot \nabla (u\eta)\, dx,
\qquad \forall \, \eta \in W^{1,\infty} (\Omega).
\end{split}
\end{equation}
By \eqref{eq:17}, for any $\eta \in C^1 (\Omega)$ we have
\begin{equation}
\begin{split}
&\lim_{i\to \infty} \int_{\Omega} \nabla u_{\varepsilon _i} \cdot \nu_{\varepsilon _i} \eta \,dx
= 
-\lim_{i\to \infty} \int_{\Omega} u_{\varepsilon _i} H_{\varepsilon_i} \eta +u_{\varepsilon_i} \nu_{\varepsilon _i} \cdot \nabla \eta \,dx \\
&= 
\lim_{i\to \infty} \int_{\Omega} (-H_{\varepsilon_i} + f_{\varepsilon_i} ) u_{\varepsilon _i} \eta -  f_{\varepsilon_i}  u_{\varepsilon _i} \eta
-u_{\varepsilon_i} \nu_{\varepsilon _i} \cdot \nabla \eta \,dx \\
&=  \int_{\Omega} u\eta \, d\mu - \int_{\Omega} u\eta \, d\mu + \int_{\Omega} \nu \cdot \nabla (u\eta) \, dx - \int_{\Omega} u \nu \cdot \nabla \eta \, dx
= \int_{\Omega} (\nu \cdot \nabla u) \eta \, dx.
\end{split}
\end{equation}
Hence 
\begin{equation}
\nabla u_{\varepsilon _i} \cdot \nu_{\varepsilon _i} \rightharpoonup \nabla u \cdot \nu
\quad \text{weakly-}\ast \ \text{in} \ L^\infty.
\end{equation}
By this and
\[
| | \nabla u_{\varepsilon_i} | + \nabla u_{\varepsilon _i} \cdot \nu_{\varepsilon _i} |
= \left| 
\frac{ |\nabla u _{\varepsilon _i}| ( \sqrt{ |\nabla u_{\varepsilon_i}|^2 + \varepsilon _i ^2 }-|\nabla u_{\varepsilon_i}| )  }{\sqrt{ |\nabla u_{\varepsilon_i}|^2 + \varepsilon _i ^2 }}
\right| \leq \varepsilon _i,
\]
we have
\begin{equation}\label{eq:20}
| \nabla u_{\varepsilon_i} | \rightharpoonup -\nabla u \cdot \nu
\quad \text{weakly-}\ast \ \text{in} \ L^\infty.
\end{equation}
By \eqref{young1} and \eqref{young2}, 
\begin{equation}\label{young3}
| \nabla u_{\varepsilon _i} | \rightharpoonup \int_{\Omega} | \lambda | \, d \nu_x (\lambda) \quad \text{weakly-}\ast \ \text{in} \ L^\infty
\end{equation}
and
\begin{equation}\label{young4}
 \nabla u_{\varepsilon _i}  \rightharpoonup \int_{\Omega} \lambda  \, d \nu_x (\lambda) \quad \text{weakly-}\ast \ \text{in} \ L^\infty,
\end{equation}
by taking $F (\lambda)=|\lambda|$ or $F(\lambda)=\lambda_k$ for $\lambda =(\lambda_1,\dots,\lambda_d)$. By the assumption
$\nabla u_{\varepsilon _i}  \rightharpoonup \nabla u$ weakly-$\ast$ in $L^\infty$, we have
\begin{equation}\label{eq:23}
\nabla u (x) =\int_{\Omega} \lambda  \, d \nu_x (\lambda),\qquad \text{a.e.} \ x \in \Omega.
\end{equation}
By this and \eqref{eq:20}, \eqref{young3}, 
\begin{equation}\label{eq:31}
-\int_{\Omega} \lambda  \, d \nu_x (\lambda) \cdot \nu (x)
=-\nabla u (x) \cdot \nu (x) = \int_{\Omega} |\lambda| \, d \nu _x (\lambda),\qquad \text{a.e.} \ x \in \Omega.
\end{equation}
Hence
\begin{equation}
\begin{split}
\int_{\Omega} |\lambda| \, d \nu _x (\lambda)
= \left| -\int_{\Omega} \lambda  \, d \nu_x (\lambda) \cdot \nu (x) \right| 
\leq \left| \int_{\Omega}  \lambda  \, d \nu_x (\lambda)\right| |\nu (x)| 
\leq \int_{\Omega} | \lambda | \, d \nu_x (\lambda) \qquad \text{a.e.} \ x \in \Omega,
\end{split}
\end{equation}
where we used $|\nu (x)| \leq 1$ for a.e.\ $x \in \Omega$. Therefore for a.e.\ $x \in \Omega$,
\begin{equation}\label{eq:33}
\int_{\Omega} | \lambda | \, d \nu_x (\lambda)
=\left| \int_{\Omega}  \lambda  \, d \nu_x (\lambda)\right|
=| \nabla u (x) |, \qquad \text{a.e.} \ x \in \Omega,
\end{equation}
where we used \eqref{eq:23}. By this and \eqref{young3}, 
$| \nabla u_{\varepsilon_i} | \rightharpoonup |\nabla u|$ weakly-$\ast$ in $L^\infty (\Omega )$. Note that this convergence holds for any subsequence $\varepsilon _i \to 0$, hence 
we obtain the statement.
\end{proof}

The following corollary shows that, thanks to the previous theorem, energy convergence implies that the potential energy vanishes in the limit $\varepsilon \to0$.

\begin{corollary}
    Suppose that $\lim_{\varepsilon \to 0} E_{\varepsilon} (u_\varepsilon) =E(u)$. Then
    $\lim_{\varepsilon \to 0} \int_{\Omega} V_{\varepsilon} (u_\varepsilon) \, dx =0$.
\end{corollary}

\begin{proof}
By the assumption and $\big||\nabla u_\varepsilon|_\varepsilon -|\nabla u_\varepsilon|\big| \leq \varepsilon$,
\[
\lim_{\varepsilon \to 0} 
\int_{\Omega} |\nabla u_\varepsilon| +V_\varepsilon (u_\varepsilon) \, dx
= \int_{\Omega} |\nabla u| \, dx.
\]
On the other hand, from Theorem \ref{thm:2.7} we have
\[
\lim_{\varepsilon \to 0} \int_{\Omega} |\nabla u_\varepsilon| \, dx
= \int_{\Omega} |\nabla u| \, dx.
\]
Therefore $\lim_{\varepsilon \to 0} \int_{\Omega} V_{\varepsilon} (u_\varepsilon) \, dx =0$.
\end{proof}

The proof of the following theorem is exactly the same as that in \cite{MR1315658}. 
\begin{theorem}\label{thm:2.9}
\begin{enumerate}[(i)]
    \item \label{thm2.9:i} Suppose $\nu_{\varepsilon_i} \rightharpoonup \nu$ weakly-$\ast$ in $L^\infty (\R^d;\R^d)$. Then we have
    \begin{equation}
    \nu (x)= -\frac{\nabla u (x)}{|\nabla u (x)|}
    \qquad \text{a.e.} \ x \in \{ |\nabla u| \not =0 \}.
    \end{equation}
    \item\label{thm2.9:ii} It holds that
    \begin{equation}
    \nu_{\varepsilon}=-\frac{\nabla u_{\varepsilon}}{|\nabla u_{\varepsilon}|_{\varepsilon}}
    \to -
    \frac{\nabla u}{|\nabla u|}
    \qquad 
    \text{strongly in} \ L^2 (\{ |\nabla u| \not =0 \} ; \R^d).
    \end{equation}
\end{enumerate}
\end{theorem}

\begin{proof}
First we denote
\[
\overline{\lambda} (x) := \int_{\Omega} \lambda \, d\nu_{x} (\lambda).
\]
By \eqref{eq:23} and \eqref{eq:31}, we have $\overline{\lambda} (x) = \nabla u (x)$ 
and 
$-\overline{\lambda} (x) \cdot \nu (x) = \int_{\Omega} |\lambda| \, d\nu_x (\lambda) $
for a.e.\ $x \in \Omega$.
Hence,
\[
\int_{\Omega} |\lambda| \, d\nu_x (\lambda)
=-\overline{\lambda} (x) \cdot \nu (x) \leq 
|\overline{\lambda} (x)| \leq 
\int_{\Omega} |\lambda| \, d\nu_x (\lambda)
\qquad \text{for a.e.} \ x \in \Omega.
\]
Therefore $|\nu|=1$ and
$\overline{\lambda} $
has the same orientation as $ \nu $ a.e.\ on $\{ \overline{\lambda} \not =0 \} = 
\{ |\nabla u| \not =0 \} $. Thus,
\[
\nu =- \frac{\overline{\lambda}}{|\overline{\lambda}|}
= -\frac{\nabla u}{|\nabla u|}
\qquad \text{for a.e.\ on} \ \{|\nabla u|\not =0\}.
\]
This shows~\eqref{thm2.9:i}. By using~\eqref{thm2.9:i}, 
\[
\nu_{\varepsilon_i} \rightharpoonup -\frac{\nabla u}{|\nabla u|}
\qquad \text{weakly-}\ast \ \text{on} \ 
\{|\nabla u|\not = 0\}
\]
for any subsequence which converges weakly. 
Hence it holds that
\[
\nu_{\varepsilon} \rightharpoonup -\frac{\nabla u}{|\nabla u|}
\qquad \text{weakly-}\ast \ \text{on} \
\{|\nabla u|\not = 0\}.
\]
We have
\[
\int_{\{|\nabla u|\not =0\}} 
|\nu_\varepsilon -\nu|^2 \, dx
=
\int_{\{|\nabla u|\not =0\}} 
(|\nu_\varepsilon|^2 + |\nu|^2 -2 \nu \cdot\nu_\varepsilon) \, dx
\leq
2\int_{\{|\nabla u|\not =0\}} 
(1 - \nu \cdot \nu_\varepsilon) \, dx.
\]
Note that since $\nu \in L^1 (\Omega; \R^d)$, 
\[
\int_{\{|\nabla u|\not =0\}} 
(1 - \nu \cdot \nu_\varepsilon) \, dx
\to 
\int_{\{|\nabla u|\not =0\}} 
(1 - |\nu|^2) \, dx
=0.
\]
Therefore we obtain~\eqref{thm2.9:ii}.
\end{proof}

\section{Existence theorem}
In this section, we assume that $\{u_{\varepsilon_i}\}_{i=1} ^\infty$ and $u$ satisfy all the claims of
Theorem \ref{thm:2.5}.
Summarizing what was obtained in the previous section, 
the following holds:
By the boundedness 
\[
\sup_{i\geq 1, (x,t) \in \Omega \times [0,T]} \{ |u_{\varepsilon_i} (x,t)|, |\nabla u_{\varepsilon_i} (x,t)| \}<\infty,
\] 
for any $t\in [0,T]$ and for any $k=1,2,\dots,d$, there exists a subsequence $\varepsilon_{i_j} \to 0$ and $w_k \in L^\infty (\Omega)$ (depending on $t$ and $k$) such that $\partial_{x_k} u_{\varepsilon_{i_j}} (\cdot,t) \rightharpoonup w_k$ weakly-$\ast$ in $L^\infty (\Omega)$. Since $u_{\varepsilon_{i_j}} (\cdot,t)$ converges to $u (\cdot,t)$ on $\Omega$ uniformly,
\begin{equation}\label{eq:35}
\int_{\Omega} w_k \zeta \, dx
= -\int_\Omega u (\cdot,t) \partial_{x_k} \zeta \, dx
\qquad \text{for any} \ \zeta \in C^1(\Omega).
\end{equation}
Hence $w=(w_1,\dots,w_d)$ is uniquely determined and $w= \nabla u (\cdot,t)$ by \eqref{eq:35} (note that $\nabla u (\cdot,t) \in L^\infty (\Omega;\R^d)$ exists for a.e.\ $t \in (0,T)$ since $u$ is Lipschitz, but \eqref{eq:35} shows the weak derivative $\nabla u (\cdot,t) \in L^\infty (\Omega;\R^d)$ exists for any $t\in [0,T]$). Therefore
$\nabla u_{\varepsilon_{i}} (\cdot,t) \rightharpoonup \nabla u(\cdot,t)$ weakly-$\ast$ in $L^\infty (\Omega)$ in the sense of full limit, for any $t \in [0,T]$
(this property corresponds to (A3) in the previous subsection).
Therefore Theorem \ref{thm:2.7} and Theorem \ref{thm:2.9} imply
\begin{equation}\label{eq:weak-star37}
|\nabla u_\varepsilon (\cdot,t)| \rightharpoonup |\nabla u(\cdot,t)| \qquad \text{weakly-}\ast \ \text{in} \ L^\infty (\Omega)
\end{equation}
for any $t\geq 0$ and
\begin{equation}
\frac{\nabla u_\varepsilon (\cdot, t)}{|\nabla u_\varepsilon (\cdot ,t) |_\varepsilon}
\to \frac{\nabla u (\cdot,t)}{|\nabla u (\cdot,t)|}
\qquad \text{in} \ L^2 (\{ x \in \Omega \mid |\nabla u (x,t)| >0 \})
\end{equation}
for any $t\geq 0$ (we write $\varepsilon_i$ as $\varepsilon$ for simplicity). Note that one can easily check that
\begin{equation}\label{eq33}
|\nabla u_\varepsilon| \rightharpoonup |\nabla u| \qquad \text{weakly-}\ast \ \text{in} \ L^\infty (\Omega\times (0,\infty))
\end{equation}
and
\begin{equation}\label{eq:36}
\frac{\nabla u_\varepsilon }{|\nabla u_\varepsilon  |_\varepsilon}
\to \frac{\nabla u }{|\nabla u |}
\qquad \text{in} \ L^2 _{loc} (\{ (x,t) \in \Omega \times (0,\infty) \mid |\nabla u (x,t)| >0 \}).
\end{equation}

By the boundedness 
$\max _{(x,t) \in \Omega_T } 
| \partial_t u_\varepsilon (x,t)|
\leq
\max_{x \in \Omega } |\nabla ^2 g (x)|$ for any $T>0$,
we may choose a further subsequence so that
\begin{equation}\label{eq:37}
\partial_t u_\varepsilon 
=|\nabla u_\varepsilon|_\varepsilon ( -H_\varepsilon + f_\varepsilon)
\rightharpoonup \partial _t u
\qquad 
\text{weakly-}\ast \ \text{in} \ L^\infty (\Omega \times (0,\infty)),
\end{equation}
where $\partial_t u \in L^\infty (\Omega \times (0,\infty))$ is the weak time derivative of $u$.

We have the following lemma.

\begin{lemma}\label{lemma3.1}
\begin{enumerate}[(i)]
\item \label{lemma3.1(i)} For any $t_1,t_2 \in [0,\infty)$ with $t_1<t_2$, we have
\begin{equation}\label{eq:38}
\int_{\Omega}
|\nabla u_\varepsilon|_\varepsilon + V_\varepsilon (u_\varepsilon) \, dx\Big|_{t=t_2}
+
\int_{t_1} ^{t_2} \int_\Omega ( -H_\varepsilon + f_\varepsilon)^2 |\nabla u_\varepsilon|_\varepsilon 
\, dx dt
=
\int_{\Omega}
|\nabla u_\varepsilon|_\varepsilon + V_\varepsilon (u_\varepsilon) \, dx\Big|_{t=t_1}
\leq C_0. 
\end{equation}
\item\label{lemma3.1(ii)} $\partial_t u =0$ a.e.\ on $\{ (x,t) \in \Omega \times (0,\infty) \mid  |\nabla u (x,t)|=0 \}$.
\end{enumerate}
\end{lemma} 
\begin{proof}
With the same arguments as in \eqref{eq:monotone}, we obtain~\eqref{lemma3.1(i)}. 
Let $A \subset \Omega \times [0,T]$ be any measurable subset of 
$\{ (x,t) \in \Omega \times (0,\infty) \mid  |\nabla u (x,t)|=0 \}$.
We compute that
\begin{equation}\label{eq:39}
\begin{split}
\left| \int_{A} \partial _t u \, dxdt \right|
= & \, 
\left| \lim_{\varepsilon \to 0} \int_{A} ( -H_\varepsilon + f_\varepsilon) |\nabla u_\varepsilon|_\varepsilon \, dxdt \right| \\
\leq & \, \limsup_{\varepsilon \to 0} 
\left( \int_A ( -H_\varepsilon + f_\varepsilon)^2 |\nabla u_\varepsilon|_\varepsilon 
\, dx dt \right)^{\frac12}
\left( \int_A  |\nabla u_\varepsilon|_\varepsilon 
\, dx dt \right)^{\frac12} \\
\leq & \, 
C_0 ^{\frac12} \limsup _{\varepsilon \to 0} \left( \int_A  |\nabla u_\varepsilon| 
\, dx dt \right)^{\frac12} 
=
C_0 ^{\frac12} \left( \int_A  |\nabla u| 
\, dx dt \right)^{\frac12} =0,
\end{split}
\end{equation}
where we used~\eqref{lemma3.1(i)}, $|\nabla u_\varepsilon|_\varepsilon \leq |\nabla u_\varepsilon| +\varepsilon$, the boundedness of $A$, and \eqref{eq33}.
Hence we obtain~\eqref{lemma3.1(ii)}.
\end{proof}

Set 
\begin{equation}\label{velocity}
V := \begin{cases}
\dfrac{ \partial_t u }{|\nabla u |} & \text{if} \ |\nabla u|>0,\\
0 & \text{if} \ |\nabla u| =0.
\end{cases}
\end{equation}

\begin{lemma}\label{lem:3.2}
\begin{enumerate}[(i)]
\item \label{lem:3.2(i)}$\partial _t u = V|\nabla u|$ a.e.\ on $\Omega \times (0,\infty)$.
In addition, 
\begin{equation}\label{eq:40}
( -H_\varepsilon + f_\varepsilon) |\nabla u_\varepsilon|_\varepsilon
\rightharpoonup V|\nabla u|
\qquad 
\text{weakly-}\ast \ \text{in} \ L^\infty (\Omega \times (0,\infty)).
\end{equation}
\item\label{lem:3.2(ii)} For any $t_1, t_2 \in [0,\infty)$ with $t_1 <t_2$, we have
\begin{equation}\label{energy-ineq}
\int_{\Omega} |\nabla u| \, dx
\Big | _{t=t_2}
+
\int_{t_1} ^{t_2} \int_\Omega V^2 |\nabla u| \, dxdt \leq 
\int_{\Omega} |\nabla u| \, dx
\Big | _{t=t_1}.
\end{equation}
\end{enumerate}
\end{lemma}

\begin{proof}
\emph{Argument for~\eqref{lem:3.2(i)}.}
By the definition of $V$ and Lemma \ref{lemma3.1}\eqref{lemma3.1(ii)}, 
$\partial _ t u = V =0$ a.e.\ on $\{ |\nabla u| =0\}$. Thus $\partial _t u = V|\nabla u|$ a.e.\ on $\R^d \times (0,\infty)$.
By this and \eqref{eq:37}, we have \eqref{eq:40}.

\emph{Argument for~\eqref{lem:3.2(ii)}.}
Set $L= \max \{  
    \max_{x \in \Omega} |\nabla \phi (x)|,
    \max_{x \in \Omega} |\nabla \psi (x)|,    
    \max_{x \in \Omega } |\nabla g (x)| \}
$ and define $\mu (A) = \int_A |\nabla u| \, dxdt$ for any 
Borel set $A \subset \Omega \times [0,\infty)$. 
Since 
$   
    | \nabla u|
    \leq L
$ for a.e.\ $(x,t)$, $\mu$ is a Radon measure on $\Omega \times [0,\infty)$.
For any $\eta \in C_c (\Omega \times (0,\infty))$, we compute that
\begin{equation}
\begin{split}
& \left| 
\int_{t_1} ^{t_2} \int_{\Omega} \eta V |\nabla u| \, dxdt 
\right|
=
\lim_{\varepsilon \to 0} \left| 
\int_ {t_1} ^{t_2} \int_{\Omega} \eta (-H_\varepsilon + f_\varepsilon ) |\nabla u_\varepsilon |_\varepsilon \, dxdt
\right| \\
\leq & \, 
\liminf_{\varepsilon \to 0} 
\left( \int_{t_1} ^{t_2} \int_{\Omega} ( -H_\varepsilon + f_\varepsilon)^2 |\nabla u_\varepsilon|_\varepsilon 
\, dx dt \right)^{\frac12}
\left( \int_{t_1} ^{t_2} \int_{\Omega} \eta^2 |\nabla u_\varepsilon|_\varepsilon 
\, dx dt \right)^{\frac12} \\
\leq & \,  
\liminf_{\varepsilon \to 0} 
\left( \int_{t_1} ^{t_2} \int_{\Omega} ( -H_\varepsilon + f_\varepsilon)^2 |\nabla u_\varepsilon|_\varepsilon 
\, dx dt \right)^{\frac12}
\left( \int_{t_1} ^{t_2} \int_\Omega \eta ^2 |\nabla u| \, dx dt
\right)^{\frac12},
\end{split}
\end{equation}
where we used \eqref{eq33} and \eqref{eq:38}. Hence we obtain
\[
\int_{t_1} ^{t_2}
\int_\Omega V^2 |\nabla u| \, dxdt
\leq  
\liminf_{\varepsilon \to 0} 
\left( \int_{t_1} ^{t_2} \int_{\Omega} ( -H_\varepsilon + f_\varepsilon)^2 |\nabla u_\varepsilon|_\varepsilon 
\, dx dt \right)^{\frac12}.
\]
By this, Theorem \ref{thm:2.5}, and \eqref{energy-ineq}, we obtain~\eqref{lem:3.2(ii)}.
\end{proof}

Let $T>0$, $\delta >0$, and
\[
A_ + := \{ (x,t) \in \Omega \times [0,T] \mid u(x,t) = \psi (x) \}
\] 
and 
\[
A_+^\delta := \{ (x,t) \in \Omega \times [0,T] \mid \dist ((x,t), A_+) \leq \delta \},
\qquad B_+ ^\delta = \Omega \times [0,T] \setminus A_+^\delta.
\]
Similarly, we define
\[
A_ - := \{ (x,t) \in \Omega \times [0,T] \mid u(x,t) = \phi (x) \}
\] 
and 
\[
A_-^\delta := \{ (x,t) \in \Omega \times [0,T] \mid \dist ((x,t), A_-) \leq \delta \},
\qquad B_- ^\delta = \Omega \times [0,T] \setminus A_- ^\delta.
\]
Since $u$ is uniformly Lipschitz and $\phi (x) < \psi (x)$ for any $x \in \Omega$,
$\dist (A_+,A_-)>0$ and there exists $\delta_0 >0$ such that $A_+ ^\delta \cap A_- ^\delta =\emptyset$ for any $\delta \in (0,\delta_0)$.

\begin{lemma}\label{lemma3.3}
\begin{enumerate}[(i)]
\item \label{item:lemma3.3(i)} For any $\delta >0$ there exist $\varepsilon_0 >0$ and $\alpha>0$ such that
\begin{equation}\label{lemma3.3(i)}
u_\varepsilon (x,t) - \phi (x) \geq \alpha \quad \text{for any} \ (x,t) \in B_ - ^\delta
\quad \text{and} \quad
\psi (x) - u_\varepsilon (x,t) \geq \alpha \quad \text{for any} \ (x,t) \in B_ + ^\delta
\end{equation}
for any $\varepsilon \in (0,\varepsilon_0)$.
\item\label{item:lemma3.3(ii)} Let $\delta >0$ and $X \in C^1 (\Omega \times [0,T];\R^d)$ satisfy
\begin{equation}
    X \cdot \nabla \phi \leq 0 \quad \text{on} \ A_ - ^\delta
    \quad \text{and} \quad
    X \cdot \nabla \psi \geq 0 \quad \text{on} \ A_ + ^\delta
\end{equation}
Then for a.e.\ $t \in [0,T]$ we have
\begin{equation}\label{BVineq}
\int_{\Omega} \left\{
 \left( \delta_{ij} - \frac{ \partial_{x_i} u \partial _{x_j} u }{|\nabla u|^2} \right)
\partial_{x_j} X^i
+V\Big(\frac{-\nabla u}{|\nabla u|}\Big)  \cdot X\right\} |\nabla u| \, dx
\geq 0.
\end{equation}
\end{enumerate}
\end{lemma}
\begin{proof}
First we prove~\eqref{item:lemma3.3(i)}. Since $\overline{B_+ ^\delta}$ is compact, there exists a point $(x_1,t_1) \in \overline{B_+ ^\delta}$ such that
\[
2\alpha := \min _{(x,t) \in \overline{B_+ ^\delta}} (\psi (x) - u (x,t))
=\psi (x_1) - u (x_1,t_1).
\]
One can easily check that $\alpha>0$. Therefore $\psi (x) - u_\varepsilon (x,t) \geq \alpha \quad \text{for any} \ (x,t) \in B_ + ^\delta$ for sufficiently small $\varepsilon>0$
because $u_\varepsilon$ converges to $u$ on $\Omega \times [0,T]$ uniformly.
The rest of the claim of~\eqref{item:lemma3.3(i)} can be shown in the same way.

Next we prove~\eqref{item:lemma3.3(ii)}. We compute that
\begin{equation}
\begin{split}
I_1 := & \int_{\Omega} X\cdot \nabla u_\varepsilon
(H_\varepsilon -f_\varepsilon ) \, dx \\
= & \,
-\int_{\Omega} X\cdot \left(\frac{-\nabla u_\varepsilon}{|\nabla u_\varepsilon |_\varepsilon}\right)
H_\varepsilon |\nabla u_\varepsilon|_\varepsilon \, dx
-
\int_{\Omega} X\cdot \frac{\nabla u_\varepsilon}{|\nabla u_\varepsilon |_\varepsilon}
f_\varepsilon |\nabla u_\varepsilon|_\varepsilon \, dx
\\
=: & \, I_2 +I_3.
\end{split}
\end{equation}
Let $F \in L^\infty (0,T)$ be a non-negative function. First we claim that
\begin{equation}\label{I1}
\int_0 ^T F I_1 \, dt \to \int_0 ^T F
\int_{\Omega} V \Big(\frac{-\nabla u}{|\nabla u|}\Big) \cdot X |\nabla u| \, dx dt 
\end{equation}
and 
\begin{equation}\label{I2}
\int_0 ^T F I_2 \, dt \to -\int_0 ^T F
\int_{\Omega} \left( \delta_{ij} - \frac{ \partial_{x_i} u \partial _{x_j} u }{|\nabla u|^2} \right)
\partial_{x_j} X^i |\nabla u| \, dx dt 
\end{equation}
as $\varepsilon \to 0$.
By \eqref{eq:36} and \eqref{eq:40}, we have
\begin{equation*}
\begin{split}
&\int_0 ^T F \int _{\{ x \in \Omega \mid |\nabla u (x,t)|>0 \}} X \cdot \nabla u_\varepsilon (H_\varepsilon - f_\varepsilon )  \, dx dt\\
=& \, 
\int_0 ^T F \int _{\{ x \in \Omega \mid |\nabla u (x,t)|>0 \}} 
X \cdot \frac{-\nabla u_\varepsilon }{|\nabla u_\varepsilon |_\varepsilon}
(-H_\varepsilon + f_\varepsilon ) |\nabla u_\varepsilon |_\varepsilon \, dx dt \\
\to & \,  
\int_0 ^T F \int _{\{ x \in \Omega \mid |\nabla u (x,t)|>0 \}} X \cdot \left( \frac{-\nabla u}{|\nabla u|} \right)
V |\nabla u| \, dx dt
=
\int_0 ^T F \int _{\Omega} X \cdot \left( \frac{-\nabla u}{|\nabla u|}\right)
V |\nabla u| \, dx dt \qquad \text{as} \ \varepsilon \to 0.
\end{split}
\end{equation*}
Note that $V=0$ a.e.\ on $\{ |\nabla u|=0 \}$.
By a computation similar to \eqref{eq:39}, one can check that
\[
\left|
\int_0 ^T F \int _{\{ x \in \Omega \mid |\nabla u (x,t)|=0 \}} X \cdot \nabla u_\varepsilon (H_\varepsilon - f_\varepsilon )  \, dx dt
\right| \to 0 \qquad \text{as} \ \varepsilon \to 0.
\]
Therefore we have \eqref{I1}.

By integration by parts, 
\begin{equation}\label{eq:49}
\begin{split}
I_2 =& \, -\int_{\Omega} X^i \partial_{x_i} u_\varepsilon
\partial_{x_j} \left( \frac{\partial_{x_j} u_\varepsilon }{|\nabla u_\varepsilon|_\varepsilon} \right)   \, dx 
= 
\int_{\Omega} (\partial_{x_j} X^i \partial_{x_i} u_\varepsilon + 
X^i \partial_{x_j} \partial_{x_i} u_\varepsilon)
 \frac{\partial_{x_j} u_\varepsilon }{|\nabla u_\varepsilon|_\varepsilon}  \, dx \\
=& \,
\int_{\Omega} 
\partial_{x_j} X^i 
\frac{\partial_{x_i} u_\varepsilon \partial_{x_j} u_\varepsilon }{|\nabla u_\varepsilon|_\varepsilon}
+ 
X^i \partial_{x_i} |\nabla u_\varepsilon|_\varepsilon  \, dx 
=
\int_{\Omega} 
\partial_{x_j} X^i 
\frac{\partial_{x_i} u_\varepsilon \partial_{x_j} u_\varepsilon }{|\nabla u_\varepsilon|_\varepsilon} 
-
\partial_{x_i} X^i  |\nabla u_\varepsilon|_\varepsilon  \, dx \\
=& \, 
-\int_{\Omega} \left( \delta_{ij} -  \frac{\partial_{x_i} u_\varepsilon \partial_{x_j} u_\varepsilon }{|\nabla u_\varepsilon|_\varepsilon ^2} \right) \partial _{x_j} X^i |\nabla u_\varepsilon|_\varepsilon \, dx.
\end{split}
\end{equation}
By \eqref{eq:36}, \eqref{eq:40}, \eqref{eq:49}, and a computation similar to \eqref{eq:39},
we have
\begin{equation}
\begin{split}
\int _0 ^T F I_2 \, dt
=& 
-\int_0 ^T F \int_{\{ x \in \Omega \mid |\nabla u (x,t)|>0 \}} \left( \delta_{ij} -  \frac{\partial_{x_i} u_\varepsilon \partial_{x_j} u_\varepsilon }{|\nabla u_\varepsilon|_\varepsilon ^2} \right) \partial _{x_j} X^i |\nabla u_\varepsilon|_\varepsilon \, dx dt\\
&- 
\int_0 ^T F \int_{\{ x \in \Omega \mid |\nabla u (x,t)| = 0 \}} \left( \delta_{ij} -  \frac{\partial_{x_i} u_\varepsilon \partial_{x_j} u_\varepsilon }{|\nabla u_\varepsilon|_\varepsilon ^2} \right) \partial _{x_j} X^i |\nabla u_\varepsilon|_\varepsilon \, dx dt \\
\to &  
-\int_0 ^T F
\int_{\Omega} \left( \delta_{ij} - \frac{ \partial_{x_i} u \partial _{x_j} u }{|\nabla u|^2} \right)
\partial_{x_j} X^i |\nabla u| \, dx dt \qquad \text{as} \ \varepsilon \to 0.
\end{split}
\end{equation}
Hence we obtain \eqref{I2}.

Next we prove 
\begin{equation}\label{I3}
\liminf_{\varepsilon \to 0} \int _0 ^T F I_3 \, dt =
\liminf_{\varepsilon \to 0} \int _0 ^T F\int_{\Omega} (-X \cdot \nabla u_\varepsilon f_\varepsilon  ) \, dx dt \geq 0.
\end{equation}
We compute that
\begin{equation}
\begin{split}
&-\int_{\Omega} X \cdot \nabla u_\varepsilon f_\varepsilon   \, dx\\
=& \, \int_\Omega X \cdot \nabla u_\varepsilon \frac{4}{\varepsilon} ((u_\varepsilon -\psi)_+)^3 \, dx - 
\int_\Omega X \cdot \nabla u_\varepsilon \frac{4}{\varepsilon} ((\phi -u_\varepsilon)_+)^3 \, dx 
= :  J_1 + J_2.
\end{split}
\end{equation}
We only prove $\liminf_{\varepsilon \to 0} \int_0 ^T J_1 \, dt \geq 0$
because
$\liminf_{\varepsilon \to 0} \int_0 ^T J_2 \, dt \geq 0$ can be shown in the same way. By integration by parts, we have
\begin{equation*}
\begin{split}
J_1 = & \, \int_\Omega X \cdot \nabla (u_\varepsilon-\psi) \frac{4}{\varepsilon} ((u_\varepsilon -\psi)_+)^3 \, dx
+\int_\Omega X \cdot \nabla \psi \frac{4}{\varepsilon} ((u_\varepsilon -\psi)_+)^3 \, dx \\
= & \, 
-\int_\Omega \divergence X \frac{1}{\varepsilon} ((u_\varepsilon -\psi)_+)^4 \, dx
+\int_\Omega X \cdot \nabla \psi \frac{4}{\varepsilon} ((u_\varepsilon -\psi)_+)^3 \, dx.
\end{split}
\end{equation*}
Since
\[
\left|\int_\Omega \divergence X \frac{1}{\varepsilon} ((u_\varepsilon -\psi)_+)^4 \, dx\right|
\leq \| \divergence X \|_{L^\infty (\Omega \times [0,T])} \int_{\Omega} V_\varepsilon (u_\varepsilon) \, dx \to 0
\]
and $\sup_{t \in [0,T] } \int_\Omega V_\varepsilon (u_\varepsilon (x,t)) \, dx <\infty$,
\begin{equation}\label{J1-1}
\int_0 ^T F \int_\Omega \divergence X \frac{1}{\varepsilon} ((u_\varepsilon -\psi)_+)^4 \, dx dt
\to 0 \qquad \text{as} \ \varepsilon \to 0
\end{equation}
by the dominated convergence theorem. We compute that
\begin{equation}\label{J1-2}
\begin{split}
&\int_\Omega X \cdot \nabla \psi \frac{4}{\varepsilon} ((u_\varepsilon -\psi)_+)^3 \, dx \\
=& \, \int_{A_+ ^\delta} X \cdot \nabla \psi \frac{4}{\varepsilon} ((u_\varepsilon -\psi)_+)^3 \, dx
+ \int_{B_+ ^\delta} X \cdot \nabla \psi \frac{4}{\varepsilon} ((u_\varepsilon -\psi)_+)^3 \, dx
\geq 0
\end{split}
\end{equation}
holds for sufficiently small $\varepsilon >0$.
Here we used $X \cdot \nabla \psi \geq 0$ on $A_ + ^\delta$ and \eqref{lemma3.3(i)}.
By \eqref{J1-1} and \eqref{J1-2} we obtain \eqref{I3}.
By \eqref{I1}, \eqref{I2}, and \eqref{I3}, we have \eqref{BVineq}.
\end{proof}

\begin{theorem}\label{thm:3.4}
Suppose $\phi, \psi, g$ are well-prepared according to Definition~\ref{def:well-prepared} and let $X \in C^1 (\Omega \times [0,T])$ satisfy
    \[
    X \cdot \nabla \phi \leq 0 \quad \text{on} \ A_-
    \quad
    \text{and}
    \quad
    X\cdot \nabla \psi \geq 0 \quad \text{on} \ A_+.
    \]
Then \eqref{BVineq} holds.
\end{theorem}

\begin{proof}
First we assume that 
\begin{equation}\label{lem3.4:assumption2}
\spt X \cap \left( \cup_{i=1} ^N B_r (x_i) \cup \cup_{i=1} ^M B_r (y_i) \right)
\times [0,T] = \emptyset
\end{equation}
for some $r>0$, 
where 
\[
\{x_1,\dots,x_N\}
=\{ x \in \Omega \mid \psi (x) \leq L+l \quad \text{and} \quad
    |\nabla \psi (x)|=0 \}
\]
and
\[
\{y_1,\dots,y_M \}
=\{ x \in \Omega \mid \phi (x) \geq -(L+l) \quad \text{and} \quad
    |\nabla \phi (x)|=0 \}.
\]
By the assumptions for $\psi$ and $\phi$, one can check that 
for $r>0$, there exists $C_r >0$ such that
\begin{equation}\label{lem3.4-0}
|\nabla \psi | \geq C_r \quad \text{on} \ \{ \psi (x) \leq L+l \} \setminus \cup_{i=1} ^N B_r (x_i)
\quad \text{and} \quad 
|\nabla \phi | \geq C_r \quad \text{on} \ \{ \phi (x) \geq -(L+l) \} \setminus \cup_{i=1} ^M B_r (y_i).
\end{equation}
By the maximum principle (Proposition \ref{prop:mp}) and \eqref{lem3.4:assumption1},
it holds that $\sup_{(x,t) \in \Omega [0,\infty)} |u_\varepsilon (x,t)| \leq L $ for any $\varepsilon >0$. Therefore 
\begin{equation*}
A_+ \subset \{ \psi \leq L \} \times [0,T]
\qquad \text{and} \qquad
A_- \subset \{ \phi \geq -L \} \times [0,T] .
\end{equation*}
Hence for any $\delta \in (0,l)$, 
\begin{equation}\label{lem3.4-1}
A_+ ^\delta \subset \{ \psi \leq L +l \} \times [0,T]
\qquad \text{and} \qquad
A_- ^\delta \subset \{ \phi \geq -(L+l) \} \times [0,T] .
\end{equation}
By \eqref{lem3.4-0} and \eqref{lem3.4-1}, if $(x,t) \in A_+ ^\delta$ and $|\nabla \psi (x)| <C_r$,
then $x \in \cup_{i=1} ^N B_r (x_i)$. Therefore \eqref{lem3.4:assumption2} implies that
\begin{equation}\label{lem3.4-2}
(x,t) \in \spt X \cap A_+^\delta
\quad \Rightarrow \quad
|\nabla \psi (x) | \geq C_r
\quad
(\text{resp.} \ (x,t) \in \spt X \cap A_- ^\delta
\quad \Rightarrow \quad
|\nabla \phi (x) | \geq C_r
)
\end{equation}
Set $\omega _+ (\delta) := \min_{A_+ ^\delta} (X \cdot \nabla \psi) \leq 0$ and
$\omega _- (\delta) := \max_{A_- ^\delta} (X \cdot \nabla \phi) \geq 0$. Note that
\begin{equation}
\omega _{\pm} (\delta) \to 0 \qquad \text{as} \ \delta \to 0,
\end{equation}
by the assumption of $X$. 
Due to \eqref{lem3.4-2}, we can choose $Y_\psi, Y_\phi \in C^1 (\Omega)$
depending only on $C_r$, $\psi$, and $\phi$ such that
\begin{equation}
\nabla \psi \cdot Y_\psi \geq \frac{C_r ^2}{2} \quad \text{on} \ \spt X \cap A_+ ^\delta 
\quad \text{and} \quad
\nabla \phi \cdot Y_\phi \geq \frac{C_r ^2}{2} \quad \text{on} \ \spt X \cap A_- ^\delta
\end{equation}
for any $\delta \in (0,l)$. 
Note that we may assume that
\[
\nabla \psi \cdot Y_\psi \geq 0 \quad \text{on} \ A_+ ^\delta 
\quad \text{and} \quad
\nabla \phi \cdot Y_\phi \geq 0 \quad \text{on} \ A_- ^\delta
\]
by using some suitable cut-off function for $\spt X \cap A_\pm ^\delta$.
Choose $\delta_\ast \in (0,l)$ such that $A_+ ^{2\delta_\ast} \cap A_- ^{2\delta_\ast} =\emptyset$ and let $\eta \in C^1 (\Omega)$ be a cut-off function for $A_+ ^{\delta_\ast}$, namely,
$0\leq \eta (x)\leq 1$ for any $x\in \Omega$,
$\eta =1$ on $A_+ ^{\delta_\ast}$, and $\eta =0$ on $A_- ^{\delta _\ast}$.
Set
\[
X_\delta := X - \frac{2}{C_r ^2} (\omega_+ (\delta) \eta Y_\psi 
+ \omega _- (\delta) (1-\eta) Y_\phi) \in C^1 (\Omega \times [0,T]) \quad \text{for any} \ \delta \in (0,\delta_\ast).
\]
If $(x,t) \in A_+ ^\delta \cap \spt X$, then 
\[
X_\delta \cdot \nabla \psi 
=
\left( X - \frac{2}{C_r ^2} \omega_+ (\delta) Y_\psi \right) \cdot \nabla \psi
\geq 
\omega_+ (\delta) - \frac{2}{C_r ^2} \omega_+ (\delta) \cdot \frac{C_r ^2}{2}\geq 0.
\]
In addition, $X_\delta \cdot \nabla \psi \geq 0$ on $A_+ ^\delta \cap (\spt X)^c$. Hence
$X_\delta \cdot \nabla \psi \geq 0$ on $A_+ ^\delta$ 
(resp. $X_\delta \cdot \nabla \phi \leq 0$ on $A_- ^\delta$)
and thus $X_\delta$ satisfies 
\begin{equation*}
\int_{\Omega} \left\{ 
\left( \delta_{ij} - \frac{ \partial_{x_i} u \partial _{x_j} u }{|\nabla u|^2} \right)
\partial_{x_j} X_\delta ^i 
+
X_\delta \cdot \left(\frac{-\nabla u}{|\nabla u|}\right) V 
\right\}
|\nabla u| \, dx \geq 0
\end{equation*}
for a.e.\ $t \in [0,T]$, by Lemma \ref{lemma3.3}. Since $\omega _\pm (\delta)\to 0$ as 
$\delta \to 0$, $X_\delta \to X$ in $C^1 (\Omega \times [0,T])$. 
Therefore \eqref{BVineq} holds 
for a.e.\ $t \in [0,T]$.

Next we prove \eqref{BVineq} without \eqref{lem3.4:assumption2}.
Let $r>0$ and $\eta_r \in C^1 (\Omega)$ be a cut-off function such that
$\eta_r =1$ on $\Omega \setminus (\cup_{i=1} ^N B_{2r} (x_i) \cup \cup_{i=1} ^M B_{2r} (y_i))$
and $\eta_r=0$ on $\cup_{i=1} ^N B_{r} (x_i) \cup \cup_{i=1} ^M B_{r} (y_i)$.
Then 
\begin{equation}\label{eta-r1}
\int_{\Omega} \left\{
\left( \delta_{ij} - \frac{ \partial_{x_i} u \partial _{x_j} u }{|\nabla u|^2} \right)
\partial_{x_j} (\eta_r X ^i) 
+\eta_r X \cdot \left( \frac{-\nabla u}{|\nabla u|} \right) V
\right\} |\nabla u| \, dx \geq 0
\end{equation}
for a.e.\ $t \in [0,T]$, by the first step. Since $\eta_r \to 1$ a.e.\ $\Omega$ as $r\to 0$, 
for a.e.\ $t \in [0,T]$ we have
\begin{equation}\label{eta-r2}
\int_{\Omega} \eta_r X \cdot \left( \frac{-\nabla u}{|\nabla u|} \right) V |\nabla u| \, dx
\to
\int_{\Omega} X \cdot \left(\frac{-\nabla u}{|\nabla u|} \right) V |\nabla u| \, dx \qquad \text{as} \ r\to 0
\end{equation}
by the dominated convergence theorem. Note that $| V| |\nabla u| =|\partial _t u | $ a.e.\ $(x,t)$ and $\| \partial _t u (\cdot,t) \|_{L^\infty (\Omega)} <\infty$ for a.e.\ $t \geq 0$. 
We compute that
\begin{equation*}
\begin{split}
\int_{\Omega} |\nabla \eta_r | \, dx
= & \, \int_0 ^1 \mathcal{H}^{d-1} (\{ \eta_r =s \}) \, ds
\leq \int_0 ^1 \mathcal{H}^{d-1} (\cup_{i=1} ^N \partial B_{2r} (x_i) \cup \cup_{i=1} ^M \partial B_{2r} (y_i)) \, ds \\
\leq & \, \omega _{d-1} (M+N) (2r)^{d-1},
\end{split}
\end{equation*}
where we used the coarea formula and $\omega_{d-1} := \mathcal{H}^{d-1} (\partial B_1 (0))$.
Therefore
\begin{equation*}
\begin{split}
&\left| \int_{\Omega} \left( \delta_{ij} - \frac{ \partial_{x_i} u \partial _{x_j} u }{|\nabla u|^2} \right)
(\partial_{x_j} \eta_r) X ^i |\nabla u| \, dx \right|
=
\left| \int_{\Omega} \left( \frac{\nabla u}{|\nabla u|} \cdot \nabla \eta_r \right) 
\left( \frac{\nabla u}{|\nabla u|} \cdot X \right)  |\nabla u| \, dx \right| \\
\leq & \, 
\max _{\Omega\times [0,T]} |X| \| \nabla u \|_{L^\infty (\Omega \times [0,T])}
\int_{\Omega} |\nabla \eta_r | \, dx
\to 0 \qquad \text{as} \ r\to 0
\end{split}
\end{equation*}
and thus
\begin{equation}\label{eta-r3}
\begin{split}
&\int_{\Omega} \left( \delta_{ij} - \frac{ \partial_{x_i} u \partial _{x_j} u }{|\nabla u|^2} \right)
\partial_{x_j} (\eta_r X ^i) |\nabla u| \, dx \\
= & \, 
\int_{\Omega} \left( \delta_{ij} - \frac{ \partial_{x_i} u \partial _{x_j} u }{|\nabla u|^2} \right)
(\partial_{x_j} \eta_r) X ^i |\nabla u| \, dx
+
\int_{\Omega} \eta_r 
\left( \delta_{ij} - \frac{ \partial_{x_i} u \partial _{x_j} u }{|\nabla u|^2} \right)
\partial_{x_j} X ^i |\nabla u| \, dx \\
\to & 
\int_{\Omega} \left( \delta_{ij} - \frac{ \partial_{x_i} u \partial _{x_j} u }{|\nabla u|^2} \right)
\partial_{x_j} X ^i |\nabla u| \, dx \qquad \text{as} \ r\to 0,
\end{split}
\end{equation}
where we used the dominated convergence theorem again.
By \eqref{eta-r1}, \eqref{eta-r2}, and \eqref{eta-r3} we obtain \eqref{BVineq}.
\end{proof}

We define $\nu:=- \dfrac{\nabla u}{|\nabla u|}$ on $\{ |\nabla u|>0 \}$
and 
\[
\Gamma_t ^\gamma := \{ x\in \Omega \mid u(x,t)=\gamma \} \qquad
\text{for any} \ \gamma \in \R \ \text{and} \ t \geq 0.
\]
The following properties are obtained without using $u$ being a limit of the solution to \eqref{mcf}.
\begin{lemma}\label{lem:3.6}
Assume $u \in W^{1,\infty} _{loc} (\Omega \times (0,\infty))$. Then the following holds.
\begin{enumerate}
\item[(i)] For a.e.\ $(\gamma,t) \in \R \times (0,\infty)$, 
\begin{equation}\label{eq:68}
\mathcal{H}^{d-1} (\Gamma _t ^\gamma \cap \{ x \in \Omega \mid u \ \text{is not differential at} \ (x,t) \ \text{or} \ |\nabla u (x,t)|=0\}) =0.
\end{equation}
\item[(ii)] For a.e.\ $(\gamma,t) \in \R \times (0,\infty)$, $\mathcal{H}^{d-1} (\Gamma _t ^\gamma) <\infty$ and $\Gamma _t ^\gamma$ is countably $(d-1)$-rectifiable and $\mathcal{H}^{d-1}$-measurable.
\end{enumerate}
\end{lemma}
\begin{proof}
Set
$v (x,t) = (u(x,t),t)^T$
 for $(x,t) \in \Omega \times (0,\infty)$. We denote
\[
\nabla_x u = \nabla u = (\partial _{x_1} u ,\dots, \partial_{x_d} u)
\qquad 
\text{and}
\qquad
\nabla_{x.t} u =(\nabla_x u , \partial_{t} u).
\]
We compute that
\[
\nabla_{x,t} v
=
\begin{pmatrix}
\nabla_{x} u & \partial_t u \\
0 & 1
\end{pmatrix}
\]
and the Jacobian is given by
\[
Jv=\sqrt{\text{det} \, (\nabla_{x,t} v \circ \nabla_{x,t} v ^T) }
=
\sqrt{\text{det} \, 
\begin{pmatrix}
|\nabla _x u| ^2 +(\partial _t u)^2 & \partial_t u \\
\partial _t u & 1 
\end{pmatrix}
}
=|\nabla _x u| \qquad \text{a.e.\ in} \ \Omega \times (0,\infty).
\]
Thus the coarea formula tells us that
\begin{equation}\label{eq:69}
\int_A |\nabla _x u| \, dxdt
=
\int_A Jv \, dxdt
= \int_{\R^2} \mathcal{H}^{(d+1)-2}
(A \cap v^{-1} (\{z\})) \, dz    
\end{equation}
for any measurable set $A \subset \Omega \times (0,\infty)$.
Substituting $A=\{ u \ \text{is not differentiable or} \ |\nabla_x u|=0 \}$ into \eqref{eq:69}, 
\[
\mathcal{H}^{d-1} (A \cap v^{-1} (\{ z \}))=0
\qquad \text{for a.e.} \ z =(\gamma ,t) \in \R\times (0,\infty).
\]
Hence we obtain (i). By \eqref{eq:69}, 
\[
\mathcal{H}^{d-1} (\Gamma _t ^\gamma) <\infty
\qquad \text{for a.e.} \ z =(\gamma ,t) \in \R\times (0,\infty).
\]
Since $v$ is Lipschitz, $\Gamma_t ^\gamma$ is countably $(d-1)$-rectifiable and $\mathcal{H}^{d-1}$-measurable for a.e.\ $(\gamma,t) \in \R\times (0,\infty)$ (for example, see \cite[Theorem 2.93]{MR1857292}).
Therefore we have (ii).
\end{proof}

\begin{remark}\label{remark4.6}
\begin{enumerate}
\item One can easily check that $\partial U_t ^\gamma \subset \Gamma_t ^\gamma =\{x \in \Omega \mid u(x,t)=\gamma\}$, but $\partial U_t ^\gamma \not =\Gamma_t ^\gamma$ for a function $u$ in general. 
\item For any set $E$ of finite perimeter, $\overline{\partial^\ast E}\subset \partial E$ (see \cite[Remark 15.3]{MR2976521}). However, for this to become an equality, a modification on sets of measure zero is generally necessary. 
\item Since $u(\cdot,t)$ is a Lipchitz function for any $t\in [0,\infty)$, $U_t ^\gamma$ is of finite perimeter for a.e. $\gamma \in \R$ (see \cite[Example 13.3]{MR2976521}).
For such $(\gamma,t)$, we have $\overline{\partial ^\ast U_t ^\gamma}\subset \partial U_t ^\gamma \subset \Gamma_t ^\gamma$.
By \eqref{eq:68}, for a.e. $(\gamma,t) \in \R\times (0,\infty)$, for $\mathcal{H}^{d-1}$-a.e. $x \in \Gamma_t ^\gamma$ it holds that $x \in \partial^e U_t ^\gamma$, where $\partial ^e E$ is the essential boundary (see \cite[page 184]{MR2976521}). Indeed, for any $x\in \Omega$ with $|\nabla u (x,t)|>0$, 
\[
\lim_{r\downarrow 0} \frac{\mathcal{L}^d (U_t^\gamma \cap B_r(x))}{\omega_d r^d }
\not \in \{0,1\}
\]
if the limit exists. By Federer's theorem (see \cite[Theorem 16.2]{MR2976521}), $\mathcal{H}^{d-1} (\partial ^e U_t ^\gamma \setminus \partial^\ast U_t^\gamma)=0$.  
Therefore 
\begin{equation}\label{rem4.6-3}
\mathcal{H}^{d-1} (\Gamma _t ^\gamma \triangle \partial^\ast U_t ^\gamma)
=
\mathcal{H}^{d-1} (\Gamma _t ^\gamma \setminus \partial^\ast U_t ^\gamma)
=
\mathcal{H}^{d-1} ((\Gamma _t ^\gamma \cap \partial ^e U_t ^\gamma)\setminus \partial^\ast U_t ^\gamma)
\leq \mathcal{H}^{d-1} (\partial ^e U_t ^\gamma \setminus \partial^\ast U_t^\gamma)=0
\end{equation}
for a.e. $(\gamma,t ) \in \R\times (0,\infty)$.
\end{enumerate}
\end{remark}

\begin{theorem}
Under the same assumptions as Theorem \ref{thm:3.4}, for any $T>0$ and for a.e.\ $(\gamma ,t) \in \R \times (0,T)$ we have 
\begin{equation}\label{BVineq2}
\int_{\Gamma _t ^\gamma}
\left\{
\left( \delta_{ij} - \nu_i \nu_j \right)
\partial_{x_j} X^i +
X \cdot \nu V 
\right\} \, d \mathcal{H}^{d-1}
\geq 0
\end{equation}
for any $X \in C^1 (\Omega \times [0,T])$ with
\[
X\cdot \nabla \psi \geq 0 \quad \text{on} \ A_+
\quad
\text{and}
\quad
X \cdot \nabla \phi \leq 0 \quad \text{on} \ A_-.
\]
\end{theorem}
\begin{proof}
Let $X \in C^1 (\Omega \times [0,T])$ 
satisfy the above assumption. 
Then for any non-negative $G \in L^\infty (0,T)$, 
by Theorem \ref{thm:3.4} 
and
$\nu= - \frac{\nabla u}{|\nabla u|}$ 
we have
\begin{equation}\label{BVineq3}
0\leq 
\int_0 ^T G
\int_{\Omega} \left\{ \left( \delta_{ij} - \nu_i \nu_j \right)
\partial_{x_j} X^i 
+ X \cdot \nu V
\right\} |\nabla u| \, dxdt.
\end{equation}
For any smooth non-negative function $F :\R\to [0,\infty)$, 
substituting $F (u_\varepsilon) X$ into \eqref{BVineq3}, 
\begin{equation}\label{BVineq4}
\begin{split}
0\leq & \, 
\int_0 ^T G
\int_{\Omega} F (u_\varepsilon) \left\{ \left( \delta_{ij} - \nu_i \nu_j \right)
\partial_{x_j} X^i 
+ X \cdot \nu V
\right\} |\nabla u| \, dxdt \\
&+
\int_0 ^T G
\int_{\Omega} F' (u_\varepsilon) \left\{ \left( \delta_{ij} - \nu_i \nu_j \right)
\partial_{x_j} u_\varepsilon X^i 
\right\} |\nabla u| \, dxdt.
\end{split}
\end{equation}
Since $u_\varepsilon \to u$ uniformly on $\Omega \times [0,T]$ and
$\partial _{x_j} u_\varepsilon \rightharpoonup \partial_{x_j} u$ weakly-$\ast$ in $L^\infty (\Omega \times 0,T))$, we have
\begin{equation}\label{BVineq5}
\begin{split}
0\leq  
\int_0 ^T G
\int_{\Omega} F (u) \left\{ \left( \delta_{ij} - \nu_i \nu_j \right)
\partial_{x_j} X^i 
+ X \cdot \nu V
\right\} |\nabla u| \, dxdt,
\end{split}
\end{equation}
where we used $(\delta_{ij} - \nu _i \nu_j)\partial _{x_j} u X^i
=|\nabla u| (\delta_{ij} -\nu_i\nu_j) \nu_j X^i=0$.
Then by the coarea formula, we obtain
\begin{equation}\label{BVineq5}
\begin{split}
0\leq  
\int_0 ^T 
\int_\R G (t) F (\gamma) \left(
\int_{\Gamma_t ^\gamma} \left\{ \left( \delta_{ij} - \nu_i \nu_j \right)
\partial_{x_j} X^i 
+ X \cdot \nu V
\right\}  \, d\mathcal{H}^{d-1} \right) d\gamma dt
\end{split}
\end{equation}
for any non-negative functions
$F \in C^\infty (\R)$ and
$G \in L^\infty (0,T)$.
Thus \eqref{BVineq2} holds for a.e.\ $(\gamma,t) \in \R \times (0,T)$. Note that $\nu$ and $V=\dfrac{\partial _t u}{|\nabla u|}$ are well-defined on $\Gamma_t ^\gamma$ for a.e.\ $(\gamma,t) \in \R \times (0,T)$ by \eqref{eq:68}.
\end{proof}

\begin{theorem}\label{thm:3.8}
\begin{enumerate}
\item[(i)]
It holds that
\begin{equation}\label{eq:75}
\int_0 ^T \int_\Omega
\partial_t \zeta u \, dxdt
= 
-\int_0 ^T \int_\Omega \zeta V |\nabla u| \, dxdt
- \int_\Omega g \zeta (\cdot,0) \, dx
\end{equation}
for any $\zeta \in C ^1 (\Omega \times [0,T))$.
In addition, for a.e.\ $\gamma \in \R$, 
\begin{equation}\label{eq:77}
\int_0 ^T \int_{U_t ^\gamma}
\partial_t \zeta \, dxdt
= 
-\int_0 ^T \int_{\Gamma_t ^\gamma} \zeta V  \, d\mathcal{H}^{d-1} dt
- \int_{U_0 ^\gamma} \zeta (\cdot,0) \, dx
\end{equation}
for any $\zeta \in C ^1 (\Omega \times [0,T))$, where
$U_t ^\gamma := \{ x \in \Omega \mid u(x,t) >\gamma \}$.
\item[(ii)]
For any $t_1,t_2 \in [0,\infty)$ with $t_1 <t_2$, we have the following: For a.e.\ $\gamma \in \R$, it holds that
\begin{equation}\label{energy-ineq2}
\mathcal{H}^{d-1} (\Gamma_{t_2} ^\gamma) 
+
\int_{t_1} ^{t_2} 
\int_{\Gamma_t ^\gamma} V^2 \, d\mathcal{H}^{d-1} dt
\leq 
\mathcal{H}^{d-1} (\Gamma_{t_1} ^\gamma) .
\end{equation}
\end{enumerate}
\end{theorem}

\begin{proof}
For any $\zeta \in C ^1 (\Omega \times [0,T))$, by integration by parts,
\begin{equation}
\begin{split}
\int_0 ^T \int_\Omega
\partial_t \zeta u_\varepsilon \, dxdt
= 
-\int_0 ^T \int_\Omega \zeta \partial _t u_\varepsilon \, dxdt
- \int_\Omega g \zeta (\cdot,0) \, dx.
\end{split}
\end{equation}
By this, \eqref{eq:37}, and Lemma \ref{lem:3.2} (i), we obtain \eqref{eq:75}.
Let $F:\R\to \R$ be a smooth function with $F' (s) >0$ for any $s \in \R$. By the existence, uniqueness, and the relabeling property of viscosity solutions (see Appendix~\ref{appendix} below), 
$F \circ u$ is the unique viscosity solution of the level set flow with initial data $F \circ g$ and obstacles $F(\phi)$ and $F(\psi)$. Therefore
the solution $w_\varepsilon$
of \eqref{mcf} with initial data $w_\varepsilon (\cdot,0)=F \circ g$ converges $F \circ u $
uniformly by the uniqueness.
Set $V_{F} := \dfrac{\partial _t (F \circ u)}{|\nabla (F \circ u)|}$. 
By \eqref{eq:68},
for a.e.\ $(\gamma, t) \in \R \times (0,\infty)$, 
\begin{equation}
\label{vv}
V_{F} = 
\frac{F' (u) \partial _t u}{F' (u) |\nabla u| }=V,
\qquad 
\text{for} \ \mathcal{H}^{d-1} \text{-a.e.} \ x \in \Gamma_t ^\gamma. 
\end{equation}
Define $K:= \inf_{(x,t) \in \Omega\times [0,\infty)} u(x,t) -1$. Note that $K$ is finite number by Lemma \ref{prop:mp}.
By \eqref{eq:75} and the coarea formula,
we have
\begin{equation}\label{eq:78}
\begin{split}
\int_0 ^T \int_\Omega
\partial_t \zeta F \circ u \, dxdt
= & \, 
-\int_0 ^T \int_\Omega \zeta V_{F} |\nabla (F \circ u)| \, dxdt
- \int_\Omega F \circ g\, \zeta (\cdot,0) \, dx \\
= & \, 
-\int_0 ^T \int _\R F' (\gamma) \int_{\Gamma_t ^\gamma} \zeta V \, d\mathcal{H}^{d-1} d\gamma dt
- \int_\Omega F \circ g\, \zeta (\cdot,0) \, dx \\
= & \, 
-\int _K ^\infty F' (\gamma) \int_0 ^T \int_{\Gamma_t ^\gamma} \zeta V \, d\mathcal{H}^{d-1} dt d\gamma
- \int_\Omega F \circ g \,\zeta (\cdot,0) \, dx
\end{split}
\end{equation}
for any $\zeta \in C ^1 (\Omega \times [0,T))$.
Here we used \eqref{vv} and $\Gamma_ t ^\gamma =\emptyset$ for any $\gamma \in (-\infty,K)$.
Since
\[
F \circ u
=F (u(x,t))
=F (K) + \int _K ^{u(x,t)} F' (\gamma) \, d\gamma, 
\]
we have
\begin{equation}\label{eq:79}
\begin{split}
\int_0 ^T \int_\Omega
\partial_t \zeta F \circ u \, dxdt
= & \, 
F (K)
\int_0 ^T \int_\Omega
\partial_t \zeta \, dxdt
+\int_0 ^T \int _\Omega \partial_t \zeta\int_K ^{u(x,t)} F'(\gamma) \, d\gamma dxdt \\
=&\, -F (K) \int_\Omega \zeta(\cdot,0) \, dx
+ \int_ 0 ^T \int_ K ^\infty \int_{U_t ^\gamma} \partial_t \zeta F'(\gamma) \, dxd\gamma dt \\
=&\, -F (K) \int_\Omega \zeta(\cdot,0) \, dx
+ \int_ K ^\infty F'(\gamma) \int_ 0 ^T  \int_{U_t ^\gamma} \partial_t \zeta \, dx dt d\gamma, 
\end{split}
\end{equation}
where we used the layercake-formula and Fubini's theorem. Similarly, it holds that
\begin{equation}\label{eq:80}
- \int_\Omega F \circ g \zeta (\cdot,0) \, dx
=
- F (K) \int_\Omega \zeta (\cdot,0) \, dx
-\int_K ^\infty F' (\gamma) \int_{U_0 ^\gamma} \zeta (\cdot,0) \, dxd\gamma.
\end{equation}
Hence \eqref{eq:78}, \eqref{eq:79}, and \eqref{eq:80} imply
\begin{equation*}
\begin{split}
& \int_ K ^\infty F'(\gamma) \int_ 0 ^T  \int_{U_t ^\gamma} \partial_t \zeta \, dx dt d\gamma \\
= & \, 
-\int _K ^\infty F' (\gamma) \int_0 ^T \int_{\Gamma_t ^\gamma} \zeta V \, d\mathcal{H}^{d-1} dt d\gamma
-\int_K ^\infty F' (\gamma) \int_{U_0 ^\gamma} \zeta (\cdot,0) \, dxd\gamma .
\end{split}
\end{equation*}
By taking $F$ as a bump-function and by using the dominated convergence theorem, 
one can easily check that
\begin{equation*}
\begin{split}
\int_ a ^b \int_ 0 ^T  \int_{U_t ^\gamma} \partial_t \zeta \, dx dt d\gamma 
=  
-\int _a ^b \int_0 ^T \int_{\Gamma_t ^\gamma} \zeta V \, d\mathcal{H}^{d-1} dt d\gamma
-\int_a ^b \int_{U_0 ^\gamma} \zeta (\cdot,0) \, dxd\gamma
\end{split}
\end{equation*}
for any $a,b \in [K,\infty)$ with $a<b$.
Then by the Lebesgue-Besicovitch differentiation theorem, we obtain \eqref{eq:77} for a.e.\ $\gamma \in \R$.

Finally, we prove \eqref{energy-ineq2}. 
Let $F :\R\to \R$ be a smooth function with $F' (s) >0$ for any $s \in \R$.
Repeating the discussion above for \eqref{energy-ineq}, we obtain
\[
\int_\R F' (\gamma) 
\left(
\mathcal{H}^{d-1}(\Gamma _t ^\gamma) 
\Big | _{t=t_2}
-
\mathcal{H}^{d-1}(\Gamma _t ^\gamma) 
\Big | _{t=t_1}
+
\int_{t_1} ^{t_2} \int_{\Gamma _t ^\gamma} V^2 \, d\mathcal{H}^{d-1} dt 
\right) \, d\gamma \leq 0
\]
and hence
\[
\int_a ^b  
\left(
\mathcal{H}^{d-1}(\Gamma _t ^\gamma) 
\Big | _{t=t_2}
-
\mathcal{H}^{d-1}(\Gamma _t ^\gamma) 
\Big | _{t=t_1}
+
\int_{t_1} ^{t_2} \int_{\Gamma _t ^\gamma} V^2 \, d\mathcal{H}^{d-1} dt 
\right) \, d\gamma \leq 0
\]
holds for any $a,b \in [K,\infty)$ with $a<b$.
Therefore by the Lebesgue-Besicovitch differentiation theorem, we obtain \eqref{energy-ineq2} for a.e.\ $\gamma \in \R$.
\end{proof}

Before proving our main result, Theorem~\ref{thm:BVsol}, we establish the following preliminary continuity property.
\begin{lemma}\label{lemma3.9}
For a.e.\ $\gamma \in \R$ and a.e.\ $t \in [0,T)$, there exists $\{ t_n\}_{n=1} ^\infty$ such that $t_n \to t$ and $|U_t ^\gamma \bigtriangleup U_{t_n} ^\gamma| \to 0$.
\end{lemma}

\begin{proof}
By Theorem \ref{thm:3.8}, \eqref{eq:77} holds for a.e.\ $\gamma \in \R$. Fix such $\gamma$.
For $t_0 \in [0,T)$, set 
\[
\zeta (x,t)=\zeta_{t_0,n} (t) \eta (x), \qquad
\zeta_{t_0,n} (t) = 
\begin{cases}
1 & \text{if} \ t \in [0,t_0],\\
1-n(t-t_0) &\text{if} \ t \in [t_0, t_0 +\frac{1}{n}],\\
0 & \text{otherwise},
\end{cases}
\]
where $\eta \in C_0 ^1 (\Omega)$.
Let $J_\varepsilon$ be the standard mollifier. Substituting $J_\varepsilon \zeta$ into \eqref{eq:77} and taking the limit, we have
\begin{equation*}
\frac{1}{n} \int_{t_0} ^{t_0 +\frac{1}{n}} \int_{U_t ^\gamma} \eta \, dxdt
= 
\int_0 ^{t_0 +\frac{1}{n}} \int_{\Gamma_t ^\gamma} \zeta V  \, d\mathcal{H}^{d-1} dt
+ \int_{U_0 ^\gamma} \eta \, dx.
\end{equation*}
For $t_0, t_1 \in [0,T)$ with $t_0<t_1$, we have
\begin{equation*}
\begin{split}
&\left| \frac{1}{n} \int_{t_1} ^{t_1 +\frac{1}{n}} \int_{U_t ^\gamma} \eta \, dx dt
-
\frac{1}{n} \int_{t_0} ^{t_0 +\frac{1}{n}} \int_{U_t ^\gamma} \eta \, dx dt\right| \\
\leq & \,
\int_{t_0 +\frac{1}{n}} ^{t_1 +\frac{1}{n}} \int_{\Gamma_t ^\gamma} |\zeta V|  \, d\mathcal{H}^{d-1} dt
\leq \sqrt{t_1 -t_0}\left(
\int_{0} ^{T} \int_{\Gamma_t ^\gamma} V^2  \, d\mathcal{H}^{d-1} dt
\right)^{\frac{1}{2}}.
\end{split}
\end{equation*}
By taking the limit $n\to \infty$, for a.e.\ $t_0$ and a.e.\ $t_1$, 
\[
\left|
\int_{U_{t_1} ^\gamma} \eta \, dx
-
\int_{U_{t_0} ^\gamma} \eta \, dx
\right|
\leq 
\sqrt{t_1 -t_0}\left(
\int_{0} ^{T} \int_{\Gamma_t ^\gamma} V^2  \, d\mathcal{H}^{d-1} dt
\right)^{\frac{1}{2}},
\]
where we used the Lebesgue-Besicovitch differentiation theorem. Hence we obtain the claim.
\end{proof}

Now we are in the position to prove our main result.
\begin{proof}[Proof of Theorem~\ref{thm:BVsol}]
Let $T>0$ be any positive number.
By Proposition \ref{prop:mp} and the assumption of Theorem \ref{thm:3.4}, we have $|u(x,t)| \leq L$ for any $(x,t) \in \Omega\times [0,T]$ and hence we only need to consider $|\gamma|<L$.
The function $u$ satisfies $\phi \leq u \leq \psi$ by Theorem \ref{thm:2.5}.
Hence, for any $\gamma \in (-L,L)$ and for any $t\in [0,T]$, we have $\Phi ^\gamma \subset U_t ^\gamma \subset (\Psi ^\gamma)^c$, where $\Phi^\gamma =\{ \phi >\gamma \}$ and $\Psi^\gamma =\{\psi >\gamma\}$.
In addition, for a.e.\ $\gamma \in (-L,L)$, \eqref{BVsol:2} and \eqref{BVsol:3} hold by~\eqref{rem4.6-3}, 
 \eqref{eq:77}, and 
\eqref{energy-ineq2}.

Next we show \eqref{BVsol:5}. 
By the assumption of Theorem \ref{thm:3.4} and the implicit function theorem, for any $\gamma \in (-L,L)\setminus \{ \psi(x_1),\dots,\psi (x_N), \phi(y_1),\dots, \phi(y_M) \}$,
\begin{equation}\label{thm4.10proof1}
\partial \Psi^\gamma \ \text{and} \ \partial \Phi^\gamma \ \text{are} \ C^1, \ \text{and} \ |\nabla \psi|>0 \ \text{on} \ \partial \Psi^\gamma \ \text{and} \ |\nabla \phi|>0 \ \text{on} \ \partial \Phi^\gamma.
\end{equation}
Note that such $\gamma$, the outer unit normal vectors of the obstacles given by 
\begin{equation}\label{obstacle-normal}
\nu_{\Psi^\gamma} = \frac{\nabla \psi}{|\nabla \psi|}
\qquad
\text{and}
\qquad
\nu_{\Phi^\gamma} = -\frac{\nabla \phi}{|\nabla \phi|}.
\end{equation}
Fix $\gamma \in (-L,L)$ such that \eqref{BVsol:1}, \eqref{BVsol:2} with $\partial ^\ast U_t$ replaced by $\partial^\ast U_t ^\gamma$, \eqref{thm4.10proof1}, and \eqref{BVineq2} hold for a.e.\ $t \in [0,T]$ and for any $X \in C^1 (\Omega \times [0,T])$ with
\[
X\cdot \nabla \psi \geq 0 \quad \text{on} \ A_+
\quad
\text{and}
\quad
X \cdot (-\nabla \phi) \geq 0 \quad \text{on} \ A_-.
\]

For a contradiction, assume that $X \in C^1 (\Omega \times [0,T];\R^d)$ satisfies 
\begin{equation}\label{boundaryconditionforthm1.2}
X\cdot \nu_{\Psi ^\gamma} \geq 0 \quad \text{on} \ \overline{M^\gamma} \cap (\partial \Psi^\gamma \times [0,T])
\quad
\text{and}
\quad
X\cdot \nu_{\Phi ^\gamma} \geq 0 \quad \text{on} \ \overline{M^\gamma} \cap (\partial \Phi^\gamma \times[0,T] )
\end{equation}
and 
\begin{equation}\label{eq:mainthm:contradiction}
    \int_0 ^T \int_{\partial ^\ast U_t ^\gamma}
\left\{
\left( \delta_{ij} - \nu_i \nu_j \right)
\partial_{x_j} X^i +
X \cdot \nu V 
\right\} \, d \mathcal{H}^{d-1} dt<0,
\end{equation}
where $M^\gamma :=\{ (x,t)\in \Omega \times [0,T] \mid x \in \partial ^\ast U_t^\gamma \}$. 
Note that 
\begin{equation}\label{admissible:M}
	X\cdot \nabla \psi \geq 0 \quad \text{on } \overline{M^\gamma} \cap (\partial \Psi ^\gamma \times [0,T])
	\quad
	\text{and}
	\quad
	X \cdot (-\nabla \phi) \geq 0 \quad \text{on } \overline{M^\gamma} \cap (\partial \Phi ^\gamma \times [0,T]),
\end{equation}
where we used \eqref{obstacle-normal} and \eqref{boundaryconditionforthm1.2}.
For $\epsilon>0$ and $\delta>0$, we define
\[
D_1 (\epsilon)
:=
\{(x,t) \in \Omega \times [0,T] \mid \dist ((x,t),\overline{M^\gamma}) < \epsilon\}
\]
and
\[
D_2 (\delta)
:=
\{(x,t) \in \Omega \times [0,T] \mid \dist ((x,t), \overline{M^\gamma} \cap (\partial \Psi ^\gamma \cup \partial \Phi ^\gamma \times [0,T])) < \delta\}.
\]
Next we claim that, for given positive sequence $\delta_n \to 0$, there exists a positive sequence $\epsilon _n \to 0$ such that
\begin{equation}\label{D1D2}
    (D_1 (\epsilon_n) \setminus D_2 (\delta_n)) \cap (A_+ \cup A_-)=\emptyset.
\end{equation}
To show this, assume that for some $\delta_n$, there exist sequences $\epsilon_j \to 0$
and $\{ (x_j,t_j) \}_{j=1} ^\infty$
such that 
\[
(x_j ,t_j) 
\in
 (D_1 (\epsilon_j) \setminus D_2 (\delta_n)) \cap (A_+ \cup A_-)\qquad \text{for any} \ j\geq 1.
\]
Without loss of generality, we may assume that $(x_j,t_j) \in A_+$ for any $j\geq 1$ and there exists $(x_0,t_0) \in \overline{M^\gamma}$ such that
$(x_j,t_j) \to (x_0,t_0)$.
By $(x_j,t_j) \in D_1 (\epsilon_j)$ and $(x_j,t_j) \not \in D_2 (\delta_n)$, we have $u(x_j,t_j)\to \gamma$ and $(x_0,t_0) \not \in D_2 (\delta_n)$.
In addition, since $(x_j,t_j) \in A_+$, $u(x_j,t_j)=\psi (x_j)$ and thus we have $\gamma =\psi (x_0)$ and $(x_0,t_0) \in \partial \Psi^\gamma \times [0,T]$. 
By $(x_j,t_j)\to (x_0,t_0) \in \overline{M^\gamma} \cap (\partial \Psi^\gamma \times [0,T])$,
we have
\[
\dist ((x_j,t_j), \overline{M^\gamma} \cap (\partial \Psi^\gamma \times [0,T]))\to 0.
\]
However, this contradicts $(x_j,t_j) \not \in D_2 (\delta_n)$ for any $j\geq 1$. Therefore we obtain \eqref{D1D2}.
Let $\eta_{n}, \zeta_n \in C^1 (\Omega \times [0,T])$ be cut-off functions such that $0\leq \eta_n (x,t),\zeta_n(x,t) \leq 1$ for any $(x,t) \in \Omega\times [0,T]$, $\eta_n (x,t)=1$ if $(x,t) \in D_1 (\epsilon_n/2) $, and
$\eta _n (x,t) =0$ if
$(x,t) \not \in D_1 (\epsilon_n) $, 
and $\zeta_n (x,t)=1$ if $(x,t) \in D_2 (\delta_n) $, and
$\zeta _n (x,t) =0$ if
$(x,t) \not \in D_2 (2\delta_n) $. We compute
\begin{equation*}
\begin{split}
0>& \, \int_0 ^T \int_{\partial^\ast U_t ^\gamma}
\left\{
\left( \delta_{ij} - \nu_i \nu_j \right)
\partial_{x_j} X^i +
X \cdot \nu V 
\right\} \, d \mathcal{H}^{d-1} dt \\
=
&\int_0 ^T \int_{\partial^\ast U_t ^\gamma}
\left\{
\left( \delta_{ij} - \nu_i \nu_j \right)
\partial_{x_j} (\eta_n X^i) +
\eta_n X \cdot \nu V 
\right\} \, d \mathcal{H}^{d-1} dt \\
= & \,
\int_0 ^T \int_{\partial^\ast U_t ^\gamma} \left\{
\left( \delta_{ij} - \nu_i \nu_j \right)
\partial_{x_j} X^i _{\delta,1} +
 X_{\delta,1} \cdot \nu V 
\right\} \, d \mathcal{H}^{d-1} dt \\
&+
\int_0 ^T \int_{\partial^\ast U_t ^\gamma}
\left\{
\left( \delta_{ij} - \nu_i \nu_j \right)
\partial_{x_j} X^i _{\delta,2} +
X_{\delta,2} \cdot \nu V 
\right\} \, d \mathcal{H}^{d-1} dt\\
=: & \, K_1 +K_2,
\end{split}
\end{equation*}
where $X_{n,1}=\eta_n \zeta_n X$ and $X_{n,2} =\eta_n (1-\zeta_n) X$.
Note that by \eqref{D1D2}
and $\spt X_{n,2} \subset D_1 (\epsilon_n) \setminus D_2 (\delta_n)$, we have
\begin{equation}\label{supportX2}
(A_+ \cup A_-) \cap \spt X_{n,2}=\emptyset.
\end{equation}
Therefore
\begin{equation*}
\begin{split}
K_2= & \, \int_0 ^T \int_{\partial^\ast U_t ^\gamma}
\left\{
\left( \delta_{ij} - \nu_i \nu_j \right)
\partial_{x_j} X^i _{\delta,2} +
X_{\delta,2} \cdot \nu V 
\right\} \, d \mathcal{H}^{d-1} dt \\
=& \, \int_0 ^T \int_{\Gamma_t ^\gamma}
\left\{
\left( \delta_{ij} - \nu_i \nu_j \right)
\partial_{x_j} X^i _{\delta,2} +
X_{\delta,2} \cdot \nu V 
\right\} \, d \mathcal{H}^{d-1} dt \geq 0,
\end{split}
\end{equation*}
where we used \eqref{rem4.6-3}, \eqref{BVineq2}, and \eqref{supportX2}.
By this and \eqref{eq:mainthm:contradiction} it holds that $K_1 < -\alpha$ for some $\alpha>0$. 
By \eqref{boundaryconditionforthm1.2} and the continuity of $X$, $\nabla \psi$, and $\nabla \phi$, there exists a subsequence $\{\delta_{n}\}_{n=1} ^\infty$ (denoted by the same index) such that
\begin{equation*}
	\begin{cases}
    X\cdot \nabla \psi \geq -\frac{1}{n} & \text{if} \ \ \dist ((x,t),\overline{M^\gamma}\cap \partial \Psi^\gamma \times[0,T])\leq 2\delta_n, \\
	X \cdot (-\nabla \phi) \geq -\frac{1}{n} & \text{if} \ \ \dist((x,t), \overline{M^\gamma} \cap \partial \Phi^\gamma \times [0,T])\leq 2\delta_n.
    \end{cases}
\end{equation*}
Note that the above estimates also hold for $X_{n,1}$. By this and $\spt X_{n,1} \subset D_2 (2\delta_n)$,  
\[X_{n,1} \cdot \nabla \psi \geq -\frac{1}{n} \ \ \text{on} \ \ \spt X_{n,1} \cap A_+ 
\quad \text{and} \quad
X_{n,1} \cdot (-\nabla \phi) \geq -\frac{1}{n} \ \ \text{on} \ \ \spt X_{n,1} \cap A_-.
\]
Repeating the same argument as in the proof of Theorem \ref{thm:3.4}, for any $n\in \mathbb{N}$ there exists 
$Y_n \in C^1 (\Omega \times [0,T] ; \R^d)$ such that $\|Y_n \|_{C^1} \leq \frac{C}{n}$ for some $C>0$
and
\[
Y_n \cdot \nabla \psi \geq 0 \quad \text{on} \ A_+ 
\quad
\text{and}
\quad
Y_n \cdot (-\nabla \phi) \geq 0 \quad \text{on} \ A_-,
\]
and
\[
Y_n \cdot \nabla \psi \geq \frac{1}{n} \quad \text{on} \ A_+ \cap N_
+
\quad
\text{and}
\quad
Y_n \cdot (-\nabla \phi) \geq \frac{1}{n} \quad \text{on} \ A_- \cap N_-,
\]
where $N_+:= \{ \dist ((x,t), \partial \Psi^\gamma \times [0,T])<\delta_1 \} $ and 
$N_-:= \{ \dist ((x,t), \partial \Phi^\gamma \times [0,T])<\delta_1 \} $.
By the properties of $X_{n,1}$ and $Y_n$, and \eqref{BVineq2} we have
\[
\int_{0} ^T \int_{\partial ^\ast U_t ^\gamma}
\left\{
\left( \delta_{ij} - \nu_i \nu_j \right)
\partial_{x_j} (X_{n,1} +Y_n)^i +
(X_{n,1} +Y_n) \cdot \nu V 
\right\} \, d \mathcal{H}^{d-1} dt
\geq 0
\]
for any $n \in \mathbb{N}$.
In addition, by \eqref{BVsol:2} and $\|Y \|_{C^1} \leq \frac{C}{n}$, for sufficiently large $n$, we have $K_1 >-\frac{\alpha}{2}$. This is a contradiction. Hence we obtain \eqref{BVsol:5}.

\bigskip

Next we prove \eqref{BVsol:4}. From \eqref{rem4.6-3}, $\partial^\ast U_t^\gamma$ may be regarded as $\Gamma_t ^\gamma$.
Let $N_1 \subset \R$ be the set of all $\gamma \in (-L,L)$ for which the claim of Lemma \ref{lemma3.9} does not hold. Then $|N_1|=0$. Let $J \subset [0,T]$ be a countable dense set and let $N_2 \subset \R$ be the set of all $\gamma \in (-L,L)$ for which 
\begin{equation}\label{eq:90}
\mathcal{H}^{d-1} (\Gamma_{\tau} ^\gamma) 
+
\int_{0} ^{\tau} 
\int_{\Gamma_t ^\gamma} V^2 \, d\mathcal{H}^{d-1} dt
\leq 
\mathcal{H}^{d-1} (\Gamma_{0} ^\gamma)
\end{equation}
does not hold for some $\tau \in J$. Note that 
 $|N_2|=0$ by \eqref{energy-ineq2}.
Fix $\gamma \in (-L,L) \setminus (N_1 \cup N_2)$ and fix $T' \in (0,T)$ such that the claim of Lemma \ref{lemma3.9} is valid. Let $\{ T_n \}_{n=1} ^\infty \subset J$ be a sequence such that $T_n \downarrow T'$. Then we have
\[
\mathcal{H}^{d-1} (\Gamma_{T'} ^\gamma)
\leq 
\liminf_{n\to \infty}
\mathcal{H}^{d-1} (\Gamma_{T_n} ^\gamma)
\]
by the lower semi-continuity of the finite perimeter and 
$|U_{T'} ^\gamma \bigtriangleup U_{T_n} ^\gamma| \to 0$, and hence
\begin{equation*}
\begin{split}
\mathcal{H}^{d-1} (\Gamma_{T'} ^\gamma) 
+
\int_{0} ^{T'} 
\int_{\Gamma_t ^\gamma} V^2 \, d\mathcal{H}^{d-1} dt
\leq & \,
\liminf_{n\to \infty}
\left(
\mathcal{H}^{d-1} (\Gamma_{T_n} ^\gamma) 
+
\int_{0} ^{T_n} 
\int_{\Gamma_t ^\gamma} V^2 \, d\mathcal{H}^{d-1} dt
\right)\\
\leq & \,
\mathcal{H}^{d-1} (\Gamma_{0} ^\gamma),    
\end{split}    
\end{equation*}
where we used \eqref{eq:90}. Therefore we obtain \eqref{BVsol:1} and \eqref{BVsol:4}. 
\end{proof}

\section{Appendix}\label{appendix}
In this appendix, we consider the viscosity solutions for the level set equation~\eqref{eq:levelset_constraint}--\eqref{eq:levelset_contact} with obstacles $\phi$ and $\psi$, based on \cite{MR1100206, Mercier}. We start by defining subsolutions.

\begin{definition}
    $u \in C (\Omega \times [0,\infty)) \cap L^\infty (\Omega \times [0,\infty))$
    is called a viscosity subsolution of~\eqref{eq:levelset_constraint}--\eqref{eq:levelset_contact}
with obstacles $\phi$ and $\psi$ if 
\begin{enumerate}
    \item[(i)] $\phi (x) \leq u (x,t) \leq \psi (x)$ for any $x \in \Omega$ and $t\geq 0$, and
    \item[(ii)] for the periodically extended $u$ and for any $\eta \in C^\infty (\R^d)$, if $u-\eta$ has a local maximum at $(x_0,t_0) \in \R^d \times (0,\infty)$ and $u(x_0,t_0)>\phi (x_0)$, then 
    \begin{equation*}
    \left\{ 
\begin{array}{ll}
\partial _t \eta \leq \left( 
\delta_{ij} -\frac{\partial_{x_i} \eta \partial _{x_j} \eta}{|\nabla \eta|^2} \right)
\partial_{x_i x_j}\eta & \text{at} \ (x_0,t_0), \\
\qquad \text{if} \ \nabla \eta (x_0,t_0)\not=0, &
\end{array} \right.
    \end{equation*}
and
    \begin{equation*}
    \left\{ 
\begin{array}{ll}
\partial _t \eta \leq \left( 
\delta_{ij} -\zeta _i \zeta_j \right)
\partial_{x_i x_j}\eta & \text{at} \ (x_0,t_0), \\
\text{for some} \ \zeta \in \R^d \ \text{with} \ |\zeta| \leq 1, & \text{if} \ \nabla \eta (x_0,t_0)=0.
\end{array} \right.
    \end{equation*}
\end{enumerate}
\end{definition}

Similarly, we define supersolutions.
\begin{definition}
    $u \in C (\Omega \times [0,\infty)) \cap L^\infty (\Omega \times [0,\infty))$
     is called a viscosity supersolution of~\eqref{eq:levelset_constraint}--\eqref{eq:levelset_contact}
with obstacles $\phi$ and $\psi$ if 
\begin{enumerate}
    \item[(i)] $\phi (x) \leq u (x,t) \leq \psi (x)$ for any $x \in \Omega$ and $t\geq 0$, and
    \item[(ii)] for the periodically extended $u$ and for any $\eta \in C^\infty (\R^d)$, if $u-\eta$ has a local minimum at $(x_0,t_0) \in \R^d \times (0,\infty)$ and $u(x_0,t_0) < \psi (x_0)$, then 
    \begin{equation*}
    \left\{ 
\begin{array}{ll}
\partial _t \eta \geq \left( 
\delta_{ij} -\frac{\partial_{x_i} \eta \partial _{x_j} \eta}{|\nabla \eta|^2} \right)
\partial_{x_i x_j}\eta & \text{at} \ (x_0,t_0), \\
\qquad \text{if} \ \nabla \eta (x_0,t_0)\not=0, &
\end{array} \right.
    \end{equation*}
and
    \begin{equation*}
    \left\{ 
\begin{array}{ll}
\partial _t \eta \geq \left( 
\delta_{ij} -\zeta _i \zeta_j \right)
\partial_{x_i x_j}\eta & \text{at} \ (x_0,t_0), \\
\text{for some} \ \zeta \in \R^d \ \text{with} \ |\zeta| \leq 1, & \text{if} \ \nabla \eta (x_0,t_0)=0.
\end{array} \right.
    \end{equation*}
\end{enumerate}
\end{definition}

\begin{definition}\label{def:viscositysol}
$u \in C(\Omega \times [0,\infty))$ is called a viscosity solution of~\eqref{eq:levelset_constraint}--\eqref{eq:levelset_contact} with obstacles $\phi$ and $\psi$ if $u$ is both a viscosity subsolution and a viscosity supersolution.
\end{definition}

The following theorem  can be shown as well as the proof in \cite[Theorem 3.2]{MR1100206} (see also \cite[Proposition 1]{Mercier}).
\begin{theorem}\label{thm4.4}
Assume that $u \in C(\Omega \times [0,\infty))$ be a viscosity subsolution of~\eqref{eq:levelset_constraint}--\eqref{eq:levelset_contact} and
$v \in C(\Omega \times [0,\infty))$ be a viscosity supersolution of~\eqref{eq:levelset_constraint}--\eqref{eq:levelset_contact}
with
\[
u(x,0) \leq v(x,0) \qquad \text{for any} \ x \in \Omega.
\]
Then it holds that $u\leq v$ on $\Omega \times [0,\infty)$.
\end{theorem}

To prove Theorem \ref{thm4.4}, we use the following lemma. 
\begin{lemma}\label{Lem5.5}
Let $w:\R^d \times [0,\infty) \to \R$ be a bounded continuous function. Then there exist constants $A,B,C>0$ depending only on $\| w \|_{L^\infty (\R^d \times [0,\infty)])}$ such that the following hold:
\begin{enumerate}
\item For any $\epsilon >0$, $w_\epsilon \leq w \leq w^\epsilon$ on $\R^d \times [0,\infty)$, where 
\[
w^\epsilon (x,t)
:=
\sup_{y\in \R^d, s \in [0,\infty)} \{
w(y,s) - \epsilon ^{-1}
(|x-y|^2 + (t-s)^2)
\} 
\]
and
\[
w_\epsilon (x,t)
:=
\inf_{y\in \R^d, s \in [0,\infty)} \{
w(y,s) + \epsilon ^{-1}
(|x-y|^2 + (t-s)^2)
\} .
\]
\item For any $\epsilon >0$,
$\| w^\epsilon  \|_{L^\infty (\R^d\times [0,\infty))} \leq A$ and
$\| w_\epsilon  \|_{L^\infty (\R^d\times [0,\infty))} \leq A$.
\item If $(y,s) \in \R^d \times [0,\infty)$ satisfies 
\[
w^\epsilon (x,t)
= w(y,s) -\epsilon^{-1} (|x-y|^2 + (t-s)^2),
\]
then it holds that
\[
|x-y|\leq B\epsilon^{1/2}
\quad 
\text{and}
\quad
|t-s|\leq B\epsilon^{1/2}
\]
for any $\epsilon >0$.
The similar conclusion holds for $w_\epsilon$.
\item For any compact set $K \subset \R^d \times [0,\infty)$, $w^\epsilon \to w$ and $w_\epsilon \to w$ uniformly on $K$ as $\epsilon \downarrow 0$. 
\item $\mathrm{Lip} \, (w^\epsilon ) \leq C\epsilon ^{-1}$ and
$\mathrm{Lip} \, (w_\epsilon ) \leq C\epsilon ^{-1}$ for any $\epsilon >0$.
\item The functions
$(x,t)\mapsto w^\epsilon (x,t) +\epsilon ^{-1} (|x|^2 +t^2)$ and
$(x,t)\mapsto w_\epsilon (x,t) -\epsilon ^{-1} (|x|^2 +t^2)$ are
convex and concave, respectively.
\item Assume that for $(x_0,t_0) \in \R^d \times (B\epsilon^{1/2},\infty)$
and $(y_0,s_0) \in \R^d \times (0,\infty)$, it holds that
\[
w^\epsilon (x_0, t_0)
= w(y_0,s_0)-\epsilon ^{-1}
(|x_0 -y_0 | ^2 + (t_0 -s_0)^2)
\]
and $w$ is a viscosity subsolution at $(y_0,s_0)$ without the obstacles, 
that is, 
for any $\eta \in C^\infty (\R^d)$, if $w-\eta$ has a local maximum at $(y_0,s_0) \in \R^d \times (0,\infty)$, then 
    \begin{equation*}
    \left\{ 
\begin{array}{ll}
\partial _t \eta \leq \left( 
\delta_{ij} -\frac{\partial_{x_i} \eta \partial _{x_j} \eta}{|\nabla \eta|^2} \right)
\partial_{x_i x_j}\eta & \text{at} \ (y_0,s_0), \\
\qquad \text{if} \ \nabla \eta (y_0,s_0)\not=0, &
\end{array} \right.
    \end{equation*}
and
    \begin{equation*}
    \left\{ 
\begin{array}{ll}
\partial _t \eta \leq \left( 
\delta_{ij} -\zeta _i \zeta_j \right)
\partial_{x_i x_j}\eta & \text{at} \ (y_0,s_0), \\
\text{for some} \ \zeta \in \R^d \ \text{with} \ |\zeta| \leq 1, & \text{if} \ \nabla \eta (y_0,s_0)=0.
\end{array} \right.
    \end{equation*}
Then 
$w^\epsilon$ is a viscosity subsolution at $(x_0,t_0)$ without the obstacles.
The similar conclusion holds for $w_\epsilon$.
\end{enumerate}
\end{lemma}

See \cite[Lemma 3.1]{MR1100206} for this proof (the claim (7) differs from the original, but the proof shows this claim).
Note that since $w$ is bounded and continuous, one can easily check that $\sup$ and $\inf$ in (1) can be replaced by $\max$ and $\min$, respectively.

\begin{proof}[Proof of Theorem \ref{thm4.4}]
If the statement does not hold, then we may assume that there exist constants $T>0$ and $a>0$
such that
\[
\max_{(x,t) \in \R^d \times [0,T]}
(u(x,t)-v (x,t)) \geq a>0,
\]
by extending the functions periodically.
Then it holds that
\begin{equation}\label{comparison:eq1}
\max_{(x,t) \in \R^d \times [0,T]}
(u(x,t)-v (x,t)-\alpha t) \geq \frac{a}{2},
\end{equation}
for sufficiently small $\alpha>0$. Since $u$ and $v$ are periodic, 
\[
u^\epsilon \to u
\quad \text{and} \quad
v_\epsilon \to v
\quad 
\text{uniformly on}
\ \R^d \times [0,T],
\]
by Lemma \ref{Lem5.5} (4) (note that $u^\epsilon$ is different from $u_\varepsilon$, which appeared before this section).
By this, \eqref{comparison:eq1}, and the continuities of $u^\epsilon$ and $v_\epsilon$, for sufficiently small $\epsilon>0$, we have
\begin{equation}\label{comparison:eq2}
0\leq u^\epsilon -u\leq \frac{a}{16}, \ 
0\leq v -v_\epsilon \leq \frac{a}{16}
\ \text{on} \ \R^d \times [0,T]
\quad 
\text{and}
\quad
\max_{(x,t) \in \R^d \times [0,T]}
(u^\epsilon (x,t)-v_\epsilon (x,t)-\alpha t) \geq \frac{a}{4}.
\end{equation}
For $\delta>0$ and $x,y \in \R^d$ and $t, t+s \in [0,\infty)$, we denote
\[
F (x,y,t,s):=
u^\epsilon (x+y, t+s) - v_\epsilon (x,t) -\alpha t -\delta^{-1} (|y|^4 +s^4).
\]
By \eqref{comparison:eq2}, it holds that
\begin{equation}\label{compasiron:eq3}
\max _{(x,t), (x+y,t+s) \in \R^d \times [0,T]}
F (x,y,t,s) \geq \frac{a}{4}.
\end{equation}
Choose the points $(x_1,t_1), (x_1+y_1, t_1 +s_1) \in \R^d \times [0,T]$ such that
\begin{equation}\label{compasiron:eq4}
F(x_1,y_1,t_1,s_1)=
\max _{(x,t), (x+y,t+s) \in \R^d \times [0,T]}
F (x,y,t,s) .
\end{equation}

In this proof we only show that for sufficiently small $\epsilon >0$ and $\delta>0$,
\begin{equation}
\label{comparison:eq10}
u^\epsilon \ \text{is a viscosity subsolution near}
\ (x_1 +y_1,t_1+s_1)
\ \text{without the obstacles}
\end{equation}
and
\begin{equation}
\label{comparison:eq11}
v_\epsilon \ \text{is a viscosity supersolution near}
\ (x_1,t_1)
\ \text{without the obstacles},
\end{equation}
since by using these key properties, we obtain the contradiction (see \cite[pages 649--652]{MR1100206} for details).
Note that the argument after showing these is exactly the same as the original proof, since there is no need to consider the obstacles.

By $F(x_1,y_1 ,t_1,s_1) >0$ and Lemma \ref{Lem5.5} (2), there exists $C>0$ depending only on $\| u \|_{L^\infty}$ and $\|v\|_{L^\infty}$
such that 
\begin{equation}\label{comparison:eq5}
|y_1| , |s_1| \leq C \delta ^{1/4}.
\end{equation}
Next we show 
\begin{equation}\label{comparison:eq6}
t_1,t_1+s_1 >B\epsilon ^{1/2}
\end{equation}
for sufficiently small $\epsilon,\delta>0$, where $B>0$ is the constant given in Lemma \ref{Lem5.5}.
Assume that $0\leq t_1 < 2 B\epsilon ^{1/2}$. Then
\begin{equation*}
\begin{split}
\frac{a}{4} \leq F (x_1,y_1,t_1,s_1)
\leq u^\epsilon (x_1 + y_1, t_1 + s_1) -v_\epsilon (x_1,t_1) 
\leq u (x_1 + y_1, t_1 + s_1) -v (x_1,t_1) + \frac{a}{8},
\end{split}
\end{equation*}
where we used \eqref{comparison:eq2}.
By this,  \eqref{comparison:eq5}, and the continuities of $u$ and $v$, we may assume that
\begin{equation}
\label{comparison:eq7}
\frac{a}{16} \leq u(x_1 +y_1,t_1 +s_1) - v (x_1 + y_1,t_1 +s_1)
\quad
\text{and}
\quad
\frac{a}{16} \leq u(x_1,t_1) - v (x_1,t_1)
\end{equation}
for sufficiently small $\epsilon, \delta>0$.
But this contradicts $u(x_1,0) \leq v (x_1,0)$
for sufficiently small $\epsilon>0$, since $t_1 \to 0$ as $\epsilon \to 0$.
Hence $t_1 \geq 2 B \epsilon ^{1/2}$ and thus \eqref{comparison:eq6} holds by \eqref{comparison:eq5}.

Next we show that for sufficiently small $\epsilon >0$ and $\delta>0$
\begin{equation}
\label{comparison:eq8}
\phi(x) < u (x,t)
\qquad \text{if}
\ |(x_1 +y_1)-x|\leq B\epsilon ^{1/2}
\ \text{and} \
\ |(t_1 +s_1)-t|\leq B\epsilon ^{1/2},
\end{equation}
and
\begin{equation}
\label{comparison:eq9}
 v (x,t) <\psi (x)
\qquad \text{if}
\ |x_1 -x|\leq B\epsilon ^{1/2}
\ \text{and} \
\ |t_1 -t|\leq B\epsilon ^{1/2}.
\end{equation}
By \eqref{comparison:eq7} and $\phi (x_1+y_1) \leq v(x_1+y_1,t_1 +s_1)$, we have
$\phi (x_1+y_1) + \frac{a}{16} \leq u (x_1+y_1 , t_1 + s_1)$. Hence we obtain \eqref{comparison:eq8} for sufficiently small $\epsilon >0$ and $\delta >0$. The inequality \eqref{comparison:eq9} can be proved similarly.

We fix $\alpha,\epsilon,\delta>0$ such that all the properties mentioned above hold. Then by Lemma \ref{Lem5.5} (7), \eqref{comparison:eq6}, and \eqref{comparison:eq8}, we obtain \eqref{comparison:eq10}. Similarly one can check  \eqref{comparison:eq11}.
\end{proof}

\bigskip

Now we prove the existence of the viscosity solution of~\eqref{eq:levelset_constraint}--\eqref{eq:levelset_contact} (see \cite[Theorem 4.2]{MR1100206}).

\begin{proof}[Proof of Theorem \ref{thm:viscositysol}]
From Theorem \ref{thm4.4}, the solution of~\eqref{eq:levelset_constraint}--\eqref{eq:levelset_contact} is unique if it exists.
    By extending the functions periodically, we may assume that
\begin{equation}\label{eq:ap1}
u _{\varepsilon_i} \to u
\quad 
\text{locally uniformly in}
\ \R^d \times [0,\infty), \quad
\phi (x) \leq u (x,t) \leq \psi (x) \ \text{for any} \ (x,t) \in \R^d \times [0,\infty),
\end{equation}
where $u_{\varepsilon_i}$ and $u$ are given by Theorem \ref{thm:2.5}.

Let $\eta \in C^\infty (\R^{d+1})$ satisfy that $u-\eta$ takes a strict local maximum at $(x_0, t_0) \in \R^d \times (0,\infty)$
and $u(x_0,t_0)>\phi (x_0,t_0)$. By \eqref{eq:ap1}, we may assume that there exists $N\geq 1$ such that
\begin{equation}\label{eq:ap2}
    u_{\varepsilon _k} >\phi
    \qquad \text{near} \ (x_0,t_0) \ \text{for any} \ k\geq N,
\end{equation}
and $u_{\varepsilon_k} -\eta$ takes a local maximum at $(x_k,t_k)$ and $(x_k, t_k) \to (x_0,t_0)$ as $k\to \infty$. At the point $(x_k, t_k)$, we have
\begin{equation}\label{eq:ap3}
    \nabla (u_{\varepsilon_k} -\eta)=0, \quad
    \partial_t (u_{\varepsilon_k} -\eta)=0,
    \quad 
    \nabla ^2 (u_{\varepsilon_k} -\eta) \leq 0.
\end{equation}
In addition, at $(x_k,t_k)$ with $k\geq N$, by \eqref{eq:ap2} and \eqref{eq:ap3}, 
we compute that
\begin{equation*}
    f_{\varepsilon_k} 
= -\frac{4}{\varepsilon_k} 
\left\{
((u_{\varepsilon_k} -\psi )_+)^3
-
((\phi -u_{\varepsilon_k} )_+)^3
\right\}
=
-\frac{4}{\varepsilon_k} 
((u_{\varepsilon_k} -\psi )_+)^3
\leq 0
\end{equation*}
and thus
\begin{equation}\label{eq:ap5}
\begin{split}
    \partial _t \eta 
    =& \, \partial u_{\varepsilon_k}
    =
    \left( \delta_{ij} - \dfrac{\partial_{x_i} u_{\varepsilon_k} \partial_{x_j} u_{\varepsilon_k}}{|\nabla u_{\varepsilon_k} |^2 +\varepsilon_k ^2} \right) \partial_{x_i x_j} u_{\varepsilon_k}
+ |\nabla u_{\varepsilon_k}|_{\varepsilon_k} f_{\varepsilon_k} (u_{\varepsilon_k})\\
\leq &\, 
\left( \delta_{ij} - \dfrac{\partial_{x_i} \eta \partial_{x_j} \eta}{|\nabla \eta |^2 +\varepsilon_k ^2} \right) \partial_{x_i x_j} \eta.
\end{split}
\end{equation}
If $|\nabla \eta (x_0,t_0)|>0$, then by taking the limit of \eqref{eq:ap5}, we have
\begin{equation}\label{eq:ap6}
\begin{split}
    \partial _t \eta 
\leq  
\left( \delta_{ij} - \dfrac{\partial_{x_i} \eta \partial_{x_j} \eta}{|\nabla \eta |^2 } \right) \partial_{x_i x_j} \eta
\qquad \text{at} \ (x_0,t_0).
\end{split}
\end{equation}
On the other hand, if $|\nabla \eta (x_0,t_0)|=0$, 
we have
\begin{equation*}
\begin{split}
    \partial _t \eta 
\leq  
\left( \delta_{ij} - \zeta _i ^k \zeta _j ^k \right) \partial_{x_i x_j} \eta
\qquad \text{at} \ (x_k,t_k),
\end{split}
\end{equation*}
where 
$\zeta_i ^k :=\dfrac{\partial_{x_i} \eta }{|\nabla \eta |_{\varepsilon _k} }$. Since 
$|\zeta ^k| =|(\zeta _1 ^k,\dots,\zeta_d ^k)|\leq 1$,
by taking some subsequence (denoted by the same index), we may assume that $\zeta^k \to \zeta$ in $\R^d$. Thus we obtain 
\begin{equation}\label{eq:ap7}
\begin{split}
    \partial _t \eta 
\leq  
\left( \delta_{ij} - \zeta _i \zeta _j  \right) \partial_{x_i x_j} \eta
\qquad \text{at} \ (x_0,t_0)
\ \text{with} \ |\zeta|\leq 1
\end{split}
\end{equation}
by taking the limit $k\to \infty$.
Next, let $\eta \in C^\infty (\R^{d+1})$ satisfy that $u-\eta$ takes a local maximum at $(x_0, t_0) \in \R^d \times (0,\infty)$
and $u(x_0,t_0)>\phi (x_0,t_0)$. Set
\[
\tilde \eta (x,t):=
\eta (x,t)+ |x-x_0|^4 + |t-t_0|^4 .
\]
Then $u-\tilde \eta$ takes a strict local maximum at $(x_0,t_0)$ and $u(x_0,t_0) > \phi (x_0,t_0)$. Then one can use the argument above. Moreover, because $\partial _t \tilde \eta =\partial _t \eta$, $\nabla \tilde \eta =\nabla \eta$, and $\nabla ^2 \tilde \eta= \nabla ^2 \eta $ at $(x_0,t_0)$,
\eqref{eq:ap6} and \eqref{eq:ap7} hold.
Thus $u$ is a viscosity subsolution of~\eqref{eq:levelset_constraint}--\eqref{eq:levelset_contact}.
Similarly one can check that $u$ is a viscosity supersolution of~\eqref{eq:levelset_constraint}--\eqref{eq:levelset_contact}, and hence $u$ is the unique viscosity solution of~\eqref{eq:levelset_constraint}--\eqref{eq:levelset_contact}.
\end{proof}

The proof of the following relabeling property is almost equivalent to \cite[Theorem 2.8]{MR1100206} and is therefore omitted.
\begin{theorem}
Assume that $u$ is a viscosity solution of~\eqref{eq:levelset_constraint}--\eqref{eq:levelset_contact} with the obstacles $\phi$ and $\psi$, and $F:\R \to \R$ is a smooth function with $F'(s)>0$ for any $s \in \R$. Then $v:= F (u)$ is a viscosity solution of~\eqref{eq:levelset_constraint}--\eqref{eq:levelset_contact} with obstacles~$F(\phi)$ and $F(\psi)$.
\end{theorem}

\section*{Acknowledgments}
    This project started during discussions between the authors at Kyoto University and the University of Regensburg. The hospitality of these institutions is gratefully acknowledged.
    This project has received funding from Deutsche Forschungsgemeinschaft (DFG, German Research Foundation) through the Research Training Group 2339 IntComSin.
    The second author is supported
    by JSPS KAKENHI grants
    JP23K03180, JP23H00085, 
    JP24K00531, and JP25K00918.

\section*{Data Availability} 
Data sharing not applicable to this article as no datasets were
generated or analysed during the current study.

\section*{Declarations}
\noindent
{\bf Conflict of interest} \ 
The authors declare that they have no conflicts of interest.

\bibliographystyle{alpha}
\bibliography{sample}

\end{document}